\documentclass[11pt]{amsart}

\usepackage{amsmath}
\usepackage{amsfonts}
\usepackage{amssymb}
\usepackage{amsthm}
\usepackage{graphics}
\usepackage{amscd}
\usepackage{graphicx,epsfig}
\usepackage{color}
\usepackage[all]{xy}
\usepackage{url}

\DeclareMathOperator{\Div}{div}

\DeclareMathOperator{\Hess}{Hess}

\renewcommand{\epsilon}{\varepsilon}

\newcommand{\boF}{\mathcal{F}}
\newcommand{\boC}{\mathcal{C}}

\newcommand{\boH}{\mathcal{H}}
\newcommand{\boA}{\mathcal{A}}

\newcommand{\boO}{\mathcal{O}}
\newcommand{\boL}{\mathcal{L}}
\newcommand{\boS}{\mathcal{S}}

\newcommand{\boU}{\mathcal{U}}
\newcommand{\boZ}{\mathcal{Z}}

\newcommand{\LL}{\mathbf{L}}
\newcommand{\MM}{\mathbf{M}}
\newcommand{\bfd}{\mathbf{d}}
\newcommand{\bff}{\mathbf{f}}

\newcommand{\R}{\mathbb{R}}

\newcommand{\Z}{\mathbb{Z}}

\renewcommand{\H}{\mathbb{H}}
\renewcommand{\S}{\mathbb{S}}
\newcommand{\N}{\mathbb{N}}

\newcommand{\T}{\mathbb{T}}
\newcommand{\la}{\langle}
\newcommand{\ra}{\rangle}
\newcommand{\eps}{\varepsilon}
\newcommand{\id}{\text{id}}

\newcommand{\Ome}{\Omega}

\newtheorem{thm}{Theorem}
\newtheorem{defn}[thm]{Definition}
\newtheorem{cor}[thm]{Corollary}
\newtheorem{prop}[thm]{Proposition}
\newtheorem{lem}[thm]{Lemma}

\newcommand{\barre}[1]{\overline{#1}}
\renewcommand{\phi}{\varphi}

\newtheorem*{thm*}{Theorem}

\newtheorem{claim}{Claim}
\newtheorem*{claim*}{Claim}

\theoremstyle{remark}

\newtheorem{remarq}{Remark}
\newtheorem*{rem*}{Remark}

\newcounter{remark}

\newcounter{case}

\newcounter{construction}

\newcounter{fact}

\newcounter{step}

\usepackage{hyperref}

\newcommand{\s}{\mathfrak s}

\usepackage{hyperref}

\title{Minimal surfaces near short geodesics in hyperbolic $3$-manifolds}
\author{Laurent Mazet}
\address{Universit\'e Paris-Est, LAMA (UMR 8050), UPEC, UPEM, CNRS, 61, avenue
du G\'en\'eral de Gaulle, F-94010 Cr\'eteil cedex, France}
\email{laurent.mazet@math.cnrs.fr}

\author{Harold Rosenberg} 
\address{Instituto Nacional de Matematica Pura e Aplicada (IMPA) Estrada Dona
Castorina 110, 22460-320, Rio de Janeiro-RJ, Brazil}
\email{rosen@impa.br}

\begin{document}

\maketitle

\begin{abstract}
If $M$ is a finite volume complete hyperbolic $3$-manifold, the quantity $\boA_1(M)$ is
defined as the infimum of the areas of closed minimal surfaces in $M$. In this paper we
study the continuity property of the functional $\boA_1$ with respect to the geometric
convergence of hyperbolic manifolds. We prove that it is lower semi-continuous and even
continuous if $\boA_1(M)$ is realized by a minimal surface satisfying some hypotheses.
Understanding the interaction between minimal surfaces and short geodesics in $M$ is the
main theme of this paper
\end{abstract}

\section{Introduction}


The area of a closed minimal surface $\Sigma$ in a complete hyperbolic $3$-manifold is
bounded above by $-2\pi \chi(\Sigma)$; this follows from the Gauss equation. Finding an
optimal lower bound for the area is a more subtle question. Notice that in dimension $2$,
there is no lower bound for the length of a closed geodesic in a hyperbolic surface.
However the Margulis lemma and the monotonicity formula does give a lower bound of
$2\pi(\cosh(\bar\eps)-1)$, for the area of a properly immersed minimal surface in a
complete hyperbolic $3$-manifold; $\bar\eps$ is the Margulis constant. According to
explicit estimates of $\bar \eps$, this number is at least 0.104 \cite{Mey}.

In a previous paper the authors proved the area is at least $2\pi$ when $\Sigma$ is a
closed embedded minimal surface in a complete finite volume hyperbolic $3$-manifold of
Heegaard genus at least $6$. If $\Sigma$ is non-orientable the lower area bound is $\pi$.
Perhaps the main goal of the present paper it to introduce techniques to resolve the
remaining cases: $2\le \mathrm{Heegaard\ genus}\le 5$.


In our paper~\cite{MaRo2}, we introduce the quantity $\boA_1(M)$, where $M$ is a compact orientable $3$-manifold. If $\boO$ denotes
the collection of all smooth
orientable embedded closed minimal surfaces in $M$ and $\boU$ the collection of all
smooth non-orientable ones, $\boA_1(M)$ is defined by
$$
\boA_1(M)=\inf(\{|\Sigma|, \Sigma\in \boO\}\cup \{2|\Sigma|, \Sigma\in \boU\})
$$
so $\boA_1(M)$ gives a lower bound for the area of any minimal surface in $M$.

The main result in~\cite{MaRo2} says that $\boA_1(M)$ is the area (or twice the area) of
some minimal surface in $M$. Moreover it gives some characterization of this minimal
surface in terms of its index and its genus. 

Let $(g_i)_i$ be a sequence of smooth Riemannian metrics on $M$ which smoothly converge to
$\bar g$. Because of the characterization of the minimal surface that realizes
$\boA_1(M,g_i)$ and thanks to a compactness result by Sharp~\cite{Shar}, it can be proved
that $\liminf \boA_1(M,g_i)\ge \boA_1(M,\bar g)$.
Moreover, if $\boA_1(M,\bar g)$ is realized by a non degenerate minimal surface,
$\lim\boA_1(M,g_i)=\boA_1(M,\bar g)$. However one can produce examples where $\boA_1$ is
not upper semi-continuous (F. Morgan suggested examples of a $2$-sphere looking like a pear).

Concerning hyperbolic manifolds, our study proves that, if $M$ is
hyperbolic and its Heegaard genus is at least $6$, then $\boA_1(M)\ge 2\pi$ which
gives a universal lower bound for the area of a minimal surface in $M$. This reasoning can
be adapted to the case $M$ is a finite volume hyperbolic manifold (not necessarily compact).



In order to remove the hypothesis about the Heegaard genus, we ask the question of the
continuity of $\boA_1$ when the space of hyperbolic manifolds is endowed with the
geometric convergence topology. Here the situation is not as above where we have a
sequence of Riemannian metrics on a fixed manifold, here we have a sequence of manifolds $M_i$
with changing topologies.
Moreover, if $(M_i)_i$ is a non trivial converging sequence of
hyperbolic manifolds then $M_i$ contains a geodesic $\gamma_i$ whose length goes to $0$.
As a consequence, an important question for our study is to understand the behaviour of a minimal
surface intersecting a neighborhood of a short geodesic.

This question has been already studied by several authors. For example, Hass~\cite{Has2}
and Huang and Wang~\cite{HuWa} study the geometry of minimal surfaces near a short
geodesic in order to construct hyperbolic manifolds that fiber over the circle but such
that the fibers can not be made minimal.

Our study of minimal surfaces near short geodesics starts with a result of
Meyerhoff~\cite{Mey}. Basically it says that a short geodesic in $M$ of length $\ell$ has
a embedded tubular neighborhood $N_{R_\ell}$ of radius $R_\ell$ and $\lim_{\ell\to
0}R_\ell=+\infty$. 

We obtain two results concerning minimal surfaces in $N_{R_\ell}$. The first
one deals with stable minimal surfaces in tubular neighborhood of short geodesics
(Corollary~\ref{cor:estimarea1}).
Basically it says that such a stable minimal surface either stays far from the short geodesic
or it intersects transversely the short geodesic. Moreover in the second case, the surface
must have a very large area in the $R_\ell$ tubular neighborhood of the geodesic. 

Our second result deals with general minimal surfaces (not assumed to be 
stable)
(Proposition~\ref{prop:estimarea}). It says
that a minimal surface in the neighborhood of a short geodesic either stays very far from
the core geodesic or comes very close to it (the estimate depending on the index of the 
minimal surface). As above in the second case, we obtain a
lower bound for the area of a minimal surface coming close to the short geodesic.

Actually these two results are very similar to results we obtained with Collin and Hauswirth
in \cite{CoHaMaRo} concerning the geometry of minimal surfaces in hyperbolic cusps. In
both cases, the argument is based on the fact that the tubular neighborhoods are foliated
by equidistant tori whose diameter are small. As a consequence, an embedded minimal surface
with bounded curvature can not be tangent to these equidistant surfaces.

Once the behaviour of minimal surfaces close to short geodesics is understood, we study the
continuity of $\boA_1$. A version of our result can be stated as follows. It is similar to
the result that can be obtained for a fixed manifold with a converging sequence of
metrics.

\begin{thm*}
Let $M_i\to \barre M$ be a converging sequence of hyperbolic cusp manifolds. Then
$$
\boA_1(\barre M)\le \liminf \boA_1(M_i).
$$
If $\boA_1(\barre M)$ is not realized by the area of a stable-unstable separating minimal
surface, then
$$
\boA_1(\barre M)=\lim \boA_1(M_i).
$$
\end{thm*}

Let us recall that "stable-unstable" means that the first eigenvalue of the stability
operator is $0$. Of course one can expect that the surface that realizes $\boA_1(\barre
M)$ is never stable-unstable but we do not know how to prove this. Actually it is possible
to expect that no minimal surface in a hyperbolic manifold is stable-unstable. In fact
the above result is a combination of two
propositions: Propositions~\ref{prop:upsemi} and \ref{prop:losemi}

The main difficulty in the proof of Proposition~\ref{prop:losemi} is to be able to control
where is located a minimal surface $\Sigma_i$ that realizes
$\boA_1(M_i)$. Actually, our study of minimal surfaces near short geodesics implies that
$\Sigma_i$ can not enter into a tubular neighborhood of a short geodesic. So it stays in a
part of $M_i$ where the convergence $M_i\to \barre M$ is just the smooth convergence of
the metric tensor. Thus a compactness result by Sharp~\cite{Shar} gives the lower
semicontinuity of $\boA_1$. Concerning Proposition~\ref{prop:upsemi}, we first prove that
$\limsup \boA_1(M_i)$ is bounded. Thus if $\boA_1(\barre M)$ is not realized by a
stable-unstable separating minimal surface $\Sigma$ then $\Sigma$ can be deformed into a
minimal surface in $M_i$. This implies the second inequality.

Of course one can also think about hyperbolic manifolds with infinite volume and ask the following question.
For which class of complete hyperbolic 3-manifolds of infinite volume can one hope for an
area lower bound $2\pi$? There may not exist a closed minimal surface in $M$, but if $\boA_1(M)$ is
realized, can one expect it to be at least $2\pi$?

\medskip

The paper is organized as follows. In Section~\ref{sec:cutuend}, we recall some basic facts
about the description of cusp and tubular ends of complete finite volume hyperbolic
$3$-manifolds. Section~\ref{sec:trans} studies the geometry of minimal surfaces with bounded
curvature in tubular ends. In Section~\ref{sec:maxprinciples}, we study the general
behaviour of minimal surfaces in tubular ends. In Section~\ref{sec:minmax2} we recall some
facts about the min-max theory for minimal surfaces that we will use in the next sections.
Section~\ref{sec:minmax} is devoted to recall the work we made in~\cite{MaRo2} and how it
should be adapted to work with non compact hyperbolic manifolds. Sections~\ref{sec:upsemi}
and \ref{sec:lowsemi} are devoted to the study of the lower and upper semi-continuity of
the $\boA_1$ functional. Finally in Appendix~\ref{sec:appen}, we prove some technical
results and formulas.

\subsection*{Preliminary remarks}
Let $S$ be a smooth Riemannian surface, we will denote by $|S|$ its area.

Let $(T,d\sigma^2)$ be a flat torus. Its universal cover is a flat $\R^2$ so we have
coordinates $(x_1,x_2)$ such that the flat metric can be written $dx_1^2+dx_2^2$. Then $T$
is the quotient of $\R^2$ by some lattice $\Gamma$. We say that $(x_1,x_2)$ is an orthonormal
coordinate system on $T$. 

Moreover, we can choose $(x_1,x_2)$ such that $\Gamma$ is generated by $v_1,v_2$ where
$v_1=(a_1,0)$ and $v_2=(a_2,b_2)$. We then say that $(x_1,x_2)$ is a well oriented
orthonormal coordinate system.

We notice that if $(T,d\sigma^2)$ has diameter $\delta$ then the lattice can be generated
by vectors of length less than $2\delta$.

\section{Hyperbolic manifolds}

In this first section we recall some facts concerning the geometry of hyperbolic
$3$-manifolds with finite volume also called cusp manifolds. We refer to \cite{BePe} for
part of this description.

\subsection{The cusp and tubular ends}\label{sec:cutuend}

Let $M$ be a complete hyperbolic $3$-manifold of finite volume. For any
$\eps$ less than the Margulis constant, the manifold $M$ can be split into two parts: the
$\eps$-thick part $M_{[\eps,\infty)}$ which is connected, not empty (recall that $p\in M_{[\eps,\infty)}$ is any non null homotopic closed loop at $p$ has length at least $\eps$) and the $\eps$-thin
part which may have a finite number of
connected components. The connected components of the thin part are of two types: cusp
ends and tubular neighborhoods of closed geodesics also called tubular ends.

Cusp ends are isometric to $E_0=T\times \R_+$ endowed with a metric 
$$
g=e^{-2t}d\sigma^2+dt^2
$$
where $d\sigma^2$ is a flat metric on the $2$-torus $T$. We define $E_t=T\times[t,+\infty)$.
We notice that if $E_0$ is a component of the $\eps$-thin part then $E_t$ is a component
of the the $\delta\eps$-thin part with $e^{-2t}\le \delta\le e^{-t/2}$.


For tubular ends, let $\gamma$ be a short geodesic in $M$ and consider $c$ a
lift of $\gamma$ to $\H^3$. If $R$
is small the $R$-tubular neighborhood $N_R$ of $\gamma$ in $M$ is the quotient of the $R$
tubular neighborhood $V_R$ of $c$ in $\H^3$ by some loxodromic transformation $\tau$ of
axis $c$ (see Figure~\ref{fig:fig5}).

In order to introduce some coordinate system, let $z$ denote arclength along $c$ and let
$\vec \nu(z), \vec \tau(z)$ be parallel orthogonal unit normal vectorfields along
$\gamma$, we introduce cylindrical coordinates in $V_R$ by 
$$
F(z,\theta,r)=\exp_{c(z)}(r(\cos\theta\vec\nu(z)+\sin\theta\vec\tau(z)))
$$
In these coordinates, the hyperbolic metric is 
\begin{equation}\label{eq:metric}
g=(\cosh^2 r) dz^2+(\sinh^2r) d\theta^2+dr^2.
\end{equation}
$N_R$ can be viewed as the quotient of $M_R=\{(t,\theta,r)\in \R^2\times [0,R]\}$ by the
relations $(t,\theta,0)\sim(t,\theta',0)$, $(t,\theta,r)\sim(t,\theta+2\pi,r)$ and
$(t,\theta,r)\sim (t+\ell,\theta+\alpha,r)$ for some parameters $\ell>0$ and $\alpha$. $\ell$ is
the length of the geodesic loop $\gamma$ and $\alpha$ is called the twist
parameter of $\gamma$ (it
is the angle of the loxodromic transformation). As above, if $N_R$ is a component of
the $\eps$-thin part, then $N_r$ is a component of the $\delta\eps$-thin part
some some $\delta\in[e^{2(r-R)},e^{(r-R)/2}]$ if $R$ and $r$ are larger than some
universal constant. In the following, we denote by $S_{\bar r}=\partial N_{\bar r}$ the torus
$\{r=\bar r\}$.

\begin{figure}[h]
\begin{center}
\resizebox{0.6\linewidth}{!}{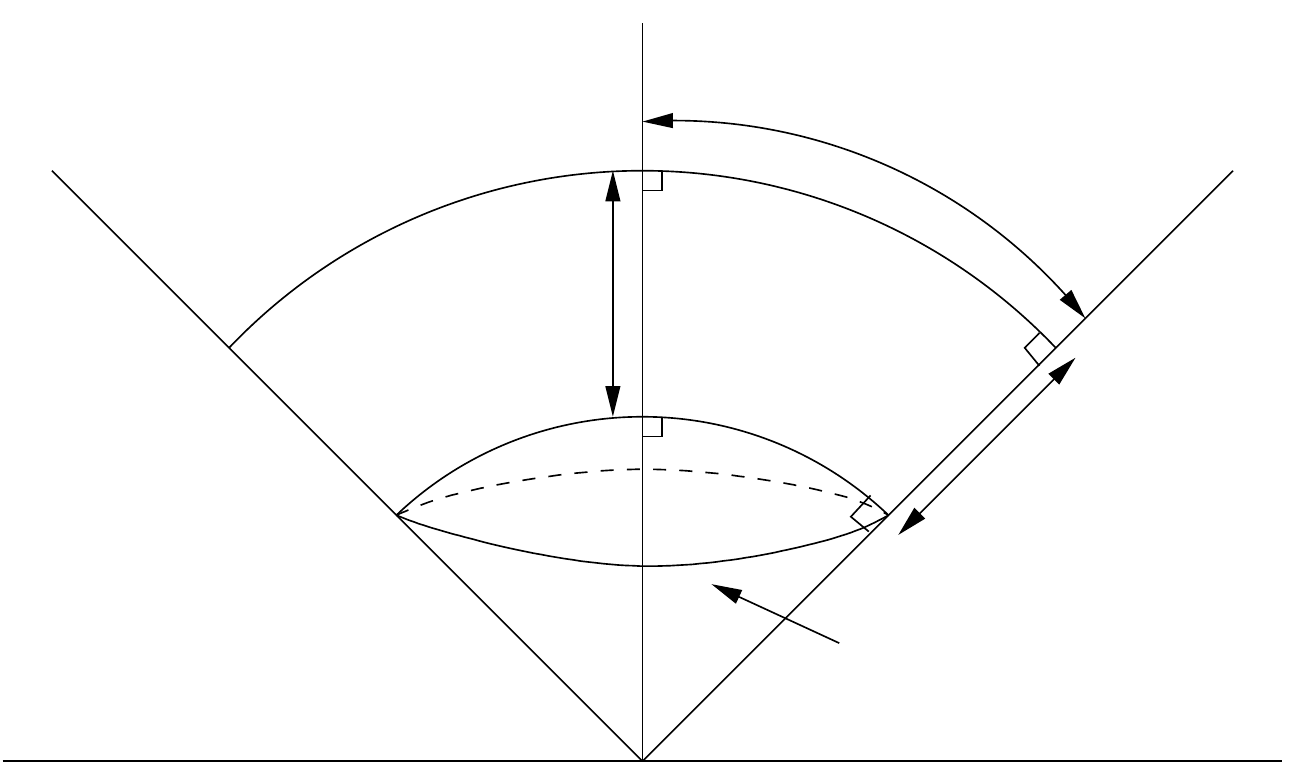}
\caption{The $r$-tubular neighborhood $V_r$ of a lift $c$ of the geodesic $\gamma$ of
length $\ell$\label{fig:fig5}}
\end{center}
\end{figure}

The above coordinates are called tubular coordinates. In order to be
coherent with the coordinates we use on cusp ends, we will also use the coordinate system
$(x_1,x_2,t)=(\theta,z,R-r)$ such that the metric can be written 
$$
g=\sinh^2(R-t)dx_1^2+\cosh^2(R-t)dx_2^2+dt^2=d\sigma_t^2+dt^2
$$
on $T\times[0,R]$ where $T$ is the quotient of $\R^2$ by the translations by $(2\pi,0)$
and $(\alpha,\ell)$ (notice that $g$ is singular on $T\times\{R\}$). 

The interest of these coordinates is that any part of a cusp or tubular end can be
described as $T\times[a,b]$ with some metric $d\sigma_t^2+dt^2$ where $d\sigma_t^2$ is a
flat metric on the torus $T$. We denote by $T_t=T\times\{t\}$. So the family $(T_t)_t$
gives a foliation of the ends by tori.

If $C$ is the torus in such an end that corresponds to $T\times\{\bar t\}$, the graph of
a function $u:\Ome\subset C\to \R$ is just the surface parametrized by $\{(p,t)\in
T\times \R|t=\bar t+u(p)\}$ (notice that we will often identify $C\subset M$ with $T_{\bar t}\in
T\times\R$).

One question is to know what is the maximal radius $R$ that can be considered
in the above discussion ($N_R$ being embedded). This has been estimated by Meyerhoff in
\cite{Mey} where the following result is proved.

\begin{thm}
Let $\gamma$ be a geodesic loop in a complete hyperbolic $3$-manifold. If the length $\ell$
of $\gamma$ is less than $\frac{\sqrt 3}{4\pi}\ln^2(\sqrt 2+1)$, then there exists an
embedded tubular neighborhood around $\gamma$ whose radius $R$ satisfies
$$
\sinh^2R=\frac12\left(\frac{\sqrt{1-2k}}k-1\right)
$$
where $k=\cosh\sqrt{\frac{4\pi \ell}{\sqrt 3}}-1$.
\end{thm}

In the sequel we denote by $R_\ell$ the solution of
$\sinh^2R=\frac12\left(\frac{\sqrt{1-2k}}k-1\right)$. When $\ell$ is small this implies that
$\sinh^2 R_\ell\sim \cosh^2 R_\ell\sim \frac{\sqrt 3}{4\pi \ell}$. For example, the area of
$S_{R_\ell}$ goes to $\frac{\sqrt 3}2$ as $\ell\to 0$.

Let us notice that the mean curvature of the torus $S_{r_0}$ with respect to $-\partial_r$
is $(\tanh r_0+\coth r_0)/2$.


\subsection{The geometric convergence}\label{sec:conv}

The space of cusp manifolds with volume less than $V_0$ is compact for geometric
convergence. This convergence is defined as follows (see Sections~E.1 and E.2 in
\cite{BePe}). Let $\Pi_i:\H^3\to M_i$ and $\barre \Pi: \H^3\to \barre M$ be the universal
covers and $o$ a point in $\H^3$. We say that the pointed manifolds $(M_i,\Pi_i(o))$
converge for the geometric convergence topology to $(\barre M,\barre \Pi(o))$ if, for any $r$,
there are $f_i:B(o,r)\subset \H^3\to \H^3$ which are
equivariant ($\barre \Pi(z)=\barre \Pi(z')\Leftrightarrow \Pi_i(f_i(z))=\Pi_i(f_i(z'))$)
such that $(f_i)_i$ converges to the identity in the $C^\infty$ topology (here $B(o,r)$
denotes the geodesic ball in $\H^3$). Actually defining $\phi_i$ by $\phi_i(\barre
\Pi(z))=\Pi_i(f_i(z))$, we will often use the following consequence (see
Lemma~E.2.2 in \cite{BePe}).

\begin{lem}
Let $(M_i)_i$ be a sequence of finite volume hyperbolic manifolds converging to $\barre M$
in the geometric topology. Let $\eps>0$ be fixed, after eliminating some initial terms,
there exists:
\begin{itemize}
\item $(\sigma_i)_i$ with $\sigma_i>0$ and $\sigma_i\to 0$,
\item $(k_i)_i$ with $k_i>1$ and $k_i\to 1$,
\item for all $i$ a $k_i$-quasi-isometric embedding $\phi_i$ from a neighborhood of
$\barre M_{[\eps,\infty)}$ into $M_i$.
\end{itemize}
with the following properties
\begin{itemize}
\item $\phi_i(\barre M_{[\eps,\infty)})$ is contained in the interior  of ${M_i}_{[\eps-\sigma_i,\infty)}$ and
\item $\phi_i(\partial\barre M_{[\eps,\infty)})$ does not meet an open neighborhood of
${M_i}_{[\eps+\sigma_i,\infty)}$.
\end{itemize}
\end{lem}

Here $k_i$-quasi isometry must be understood as smooth maps $\phi_i$ such that
$$
\frac1{k_i}d(p,q)\le d(\phi_i(p),\phi_i(q))\le k_id(p,q)
$$
When we will use these properties, we will not forget that $\phi_i$ come from maps $f_i$ that
are $C^\infty$ close to $\id$.

Actually, $\eps$ is always chosen small enough such that the $\eps$-thin part of $\barre
M$ contains only cusp ends. Moreover if $\eps$ is small enough each connected component of
${M_i}_{[\eps-\sigma_i,\eps+\sigma_i]}$ contains exactly one component of $\phi_i(\partial
\barre M_{[\eps,\infty)})$ (see Theorem~ E.2.4 in \cite{BePe}). The description of this
component of $\phi_i(\partial \barre M_{[\eps,\infty)})$ is given by the following result.

\begin{lem}\label{lem:conv2}
Let $(M_i)_i$, $\barre M$, $\eps>0$ and $\phi_i$ as above. Let $C$ be a connected
component of $\partial \barre M_{[\eps,\infty)}$. Then for $i$ large, $\phi_i(C)$ is a
graph of a function $u_i$
over the corresponding component $C_i$ of $\partial {M_i}_{[\eps,\infty)}$. Moreover
$u_i\to 0$ and $C$ and $C_i$ are $\kappa_i$-quasi-isometric with $\kappa_i\to 1$.
\end{lem}

\begin{proof}
$C$ is a surface with principal curvatures $1$. Thus $\phi_i(C)$ has principal curvatures
close to $1$ and between $1/2$ and $3/2$. 

Since $\phi_i(C)\subset A_i$ where $A_i$ is a component of
${M_i}_{[\eps-\sigma_i,\eps+\sigma_i]}$, $\phi_i(C)$ is contained either
in a cusp end of $M_i$ or a neighborhood of a short geodesic of $\gamma_i$. In the second
case, there is a smallest $\delta_i\le \sigma_i$ such that $\phi_i(C)\subset
{M_i}_{[\eps-\sigma_i,\eps+\delta_i]}$ so $\phi_i(C)$ is tangent to a
boundary torus of $\partial {M_i}_{[\eps_i+\delta_i,\infty)}$. The comparison of the mean
curvature at this tangency point gives the mean curvature of $\partial
{M_i}_{[\eps_i+\delta_i,\infty)}$ is close to $1$. Thus the distance from $\gamma_i$ to
$\phi_i(C)$ is very large and goes to $+\infty$. 

In both cases, $A_i$ is described as $T_i\times[-\alpha_i,\beta_i]$ with a metric
$d\sigma_{i,t}^2+dt^2$ with $\alpha_i,\beta_i\to0$ and $C_i=T_i\times\{0\}$ is the boundary of
${M_i}_{[\eps,\infty)}$ in $A_i$.

Let $\gamma$ be a geodesic in $\phi_i(C)$. Since $\phi_i(C)$ has curvature uniformly bounded,
there is $k_0$ such that $|\partial_s(\gamma'(s),\partial_t)|\le k_0$ so
$|(\gamma'(s),\partial_t)|\ge |(\gamma'(0),\partial_t)|/2$ for
$0<s<s_0=|(\gamma'(0),\partial_t)|/(2k_0)$. Looking at the $t$ coordinate along $\gamma$,
we then have $\beta_i+\alpha_i\ge|t(\gamma(s_0))-t(\gamma(0))|\ge
\frac{|(\gamma'(0),\partial_t)|^2}{4k_0}$. Since $\alpha_i+\beta_i\to 0$, this implies
that the angle between $\phi_i(C)$ and $\partial_t$ goes to $\pi/2$ uniformly. Since
$\phi_i(C)$ is embedded this implies that
$\phi_i(C)$ is a graph over $C_i$: there is a function $u_i:C_i\to \R$
such that $\phi_i(C)=\{(p,t)\in T\times \R|t=u_i(p)\}$.
Since the angle between $\phi_i(C)$ and $\partial_t$ goes to $\pi/2$, the
gradient of $u_i$ goes to $0$. Besides $\phi_i(C)\subset A_i$ so $|u_i|$ is close to
$0$.

This implies that $(p,u_i(p))\in\phi_i(C)\mapsto (p,0)\in C_i$ is
a  $\kappa_i$-quasi-isometry ($\kappa_i\to 1$) which can be composed with $\phi_i:C\to
\phi_i(C)$ to obtain a $\kappa_ik_i$-quasi-isometry.
\end{proof}

As consequence, we have the following result.

\begin{cor}\label{cor:diamsys}
Let $V_0$ be positive then there are $\ell_0$, $\delta_0$, $\s_0$ such the following is
true. Let $M$ be a cusp manifold with volume less than $V_0$ and $\gamma$ be a geodesic
loop of length $\ell\le \ell_0$. Then $S_{R_\ell}=\partial N_{R_\ell}$ has diameter less than $\delta_0$ and
systole larger than $\s_0$.
\end{cor}

The set of flat tori with diameter less than $\delta_0$ and systole larger than 
$\s_0$ is a compact subset of the set of flat tori.

\begin{proof}
If it not true there is a sequence of cusp manifolds $M_i$ that converge to $\barre M$
and in $M_i$ there is a geodesic loop $\gamma_i$ of length $\ell_i\to 0$ such that either the
diameter of $S_{R_{\ell_i}}$ goes to $\infty$ or its systole goes to $0$.

After taking a subsequence, we can assume that the tubular ends around $\gamma_i$
converges to one cusp end in $\barre M$. Let $\eps>0$ be small and consider $C$ the
component of $\partial \barre M_{[\eps,\infty)}$ inside this cusp end. Let $C_i$ be the
component of $\partial {M_i}_{[\eps,\infty)}$ inside the tubular end around $\gamma_i$. By
the above lemma, $C$ and $C_i$ are $2$ quasi-isometric. So the area $C_i$ is close to that
 of $C$. Since the area of $S_{R_{\ell_i}}$ in $M_i$ is close to $\sqrt 3/2$ this
implies that the distance between $C_i$ and $S_{R_{\ell_i}}$ is uniformly bounded. Since
the diameter and the systole of $S_{R_{\ell_i}}$ differ from those of $C_i$ by at most
a uniform factor. This contradicts that either the diameter goes to $\infty$ or the
systole goes to $0$.
\end{proof}

\begin{remarq}
Let us consider a particular one sided neighborhood of $\phi_i(C)$ in $M_i$.
Actually, let $\barre A$ be the part of the $2$-tubular neighborhood of $C$
inside $\barre M_{[\eps,\infty)}$. Thus $\phi_i(\barre A)$ is a one sided
neighborhood of
$\phi_i(C)$.

$\barre A$ can be parametrized by $T\times[-2,0]$ with the metric
$g=e^{-2t}d\bar\sigma^2+dt^2$. Let $X:T\times[-2,0]\to \barre M$ be this parametrization
and $(x_1,x_2)$ be orthonormal coordinates such that $g=e^{-2x_3}(dx_1^2+dx_2^2)+dx_3^2$.
Let us now
estimate the metric $\phi_i^*g_i$. We notice that $X$ lifts to an equivariant map
$\widetilde X:\R^2\times[-2,0]\to \H^3$ \textit{i.e.} $X=\barre \Pi \widetilde X$. If $\tilde
g_i=\tilde g_{i,kl}dx_kdy_l$ we have 
\begin{align*}
\tilde g_{i,kl}&=\la d\phi_i X_{y_k},d\phi_i X_{y_l} \ra_{M_i}\\
&=\la d\phi_i d\barre \Pi \widetilde X_{y_k},d\phi_i d\barre \Pi \widetilde X_{y_l} \ra_{M_i}\\
&=\la d\Pi_i df_i \widetilde X_{y_k},d\Pi_i df_i  \widetilde X_{y_l} \ra_{M_i}\\
&=\la df_i \widetilde X_{y_k}, df_i  \widetilde X_{y_l} \ra_{\H^3}
\end{align*}
since $\Pi_i$ is a local isometry. Since $f_i$ converges to the identity map in the
$C^\infty$ topology this implies that $\tilde g_i\to g$ in the
$C^\infty$ topology.
\end{remarq}

\begin{remarq}
The topology of a complete finite volume hyperbolic $3$- manifold determines its
hyperbolic structure. Thus if a converging sequence $M_i\to \barre M$ is not constant,
there is a subsequence whose topologies are distinct from that of $\barre M$. then there
are short geodesics $\gamma_i$ in $M_i$ whose lengths converge to zero and whose maximal embedded
tubular neighborhoods are converging to cusp ends of $\barre M$ (see Figure~\ref{fig:fig4}).

\begin{figure}[h]
\begin{center}
\resizebox{0.95\linewidth}{!}{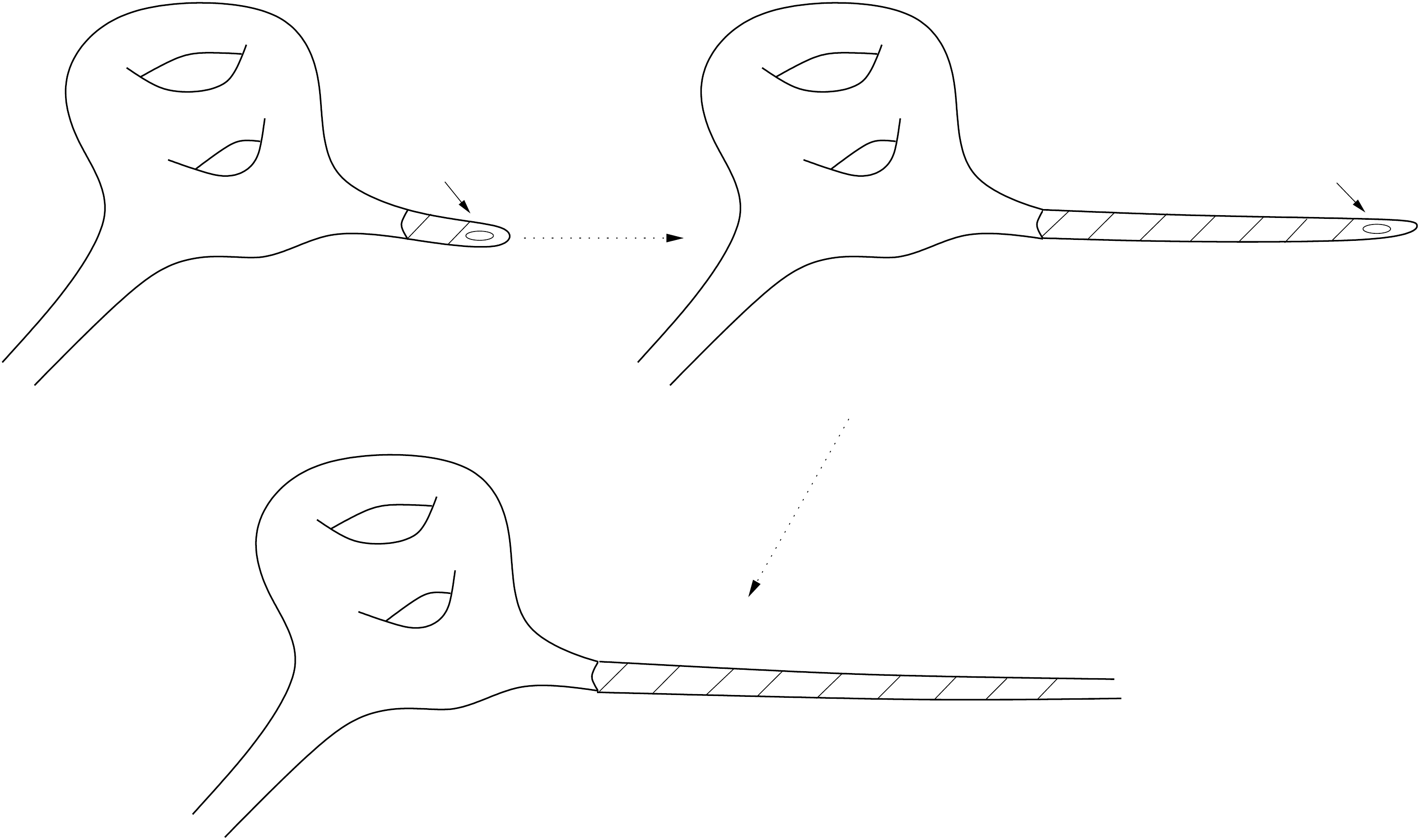}
\caption{A schematic converging sequence $M_i\to \barre M$\label{fig:fig4}}
\end{center}
\end{figure}

\end{remarq}


\section{Transversallity in tubular ends}\label{sec:trans}

The aim of this section is to understand the behaviour of a minimal surface in a tubular
end when we know \text{a priori} an upper bound on its curvature. A similar study was
made for cusp ends in \cite{CoHaMaRo}.

In this section, we use the tubular coordinates $(z,\theta,r)$.


\subsection{An intersection property}

We recall that, if $c$ is a geodesic in $\H^3$, $V_r$ denotes its tubular neighborhood of
radius $r$. Moreover, for $r>0$, we denote $B_r=\partial V_r$.

\begin{lem}\label{lem:prox}
Let $k_0$ and $\eps_0$ be positive, then there are $r_0$ and $\eta_0$ such that the
following is true. Let
$c$ be a geodesic in $\H^3$. Let $r\in[0,r_0]$ and $p_i=(z_i,\theta_i,r)$ ($i=1,2$) be
two points in $V_{r_0}$ such that $\theta_2\in[\theta_1+\frac\pi 3,\theta_1+\frac{2\pi}3]$ and
$z_2\in[z_1-\eta_0,z_1+\eta_0]$. Let $\Sigma_i$, $i=1,2$ be  surfaces in $V_{r_0}$ whose
curvatures are bounded by $k_0$, $p_i\in\Sigma_i$ and
$d_{\Sigma_i}(p_i,\partial\Sigma_i)>\eps_0$. If both $\Sigma_i$ are tangent to $B_r$ at
$p_i$ (if $r=0$ we assume moreover that a unit normal vector to $\Sigma_i$ at $p_i$ is
$\partial_r(z_i,\theta_i,0)$) then $\Sigma_1$ and $\Sigma_2$ has non empty transversal
intersection.
\end{lem}

\begin{proof}
We look for $r_0\le 2$. In $V_2$ the hyperbolic metric is $\cosh^2 r dz^2+\sinh ^2
rd\theta^2+dr^2$. Let us change
the metric in $V_2$ to the Euclidean metric $g_e=dz^2+r^2 d\theta^2+dr^2$. So there are
constants $\tilde k_0$ and $\tilde \eps_0$ depending only on $k_0$ and $\eps_0$ such
that, with $g_e$, $\Sigma_1$ and $\Sigma_2$ have curvature bounded by
$\tilde k_0$ and $d_{\Sigma_i}(p_i,\partial\Sigma_i)>\tilde\eps_0$.

Thus there is $\eta_1>0$ such that $\Sigma_i$ can be described as a graph over the Euclidean
disk of radius $\eta_1$ tangent to $\Sigma_i$ at $p_i$ (see Proposition~2.3 in
\cite{RoSoTo}). Moreover if $\eta_1$ is chosen
small enough, the gradient of the function parametrizing $\Sigma_i$ is less than
$1/10$.

Let $r_0=\eta_0=\eta_1/10$. With these choices, the tangent disks of radius $\eta_1$
tangent to $\Sigma_i$ at $p_i$ must intersect at an angle between $\pi/3$ and $2\pi/3$
(see the schematic figure~\ref{fig:fig1}). Moreover since each $\Sigma_i$ is at a distance less
than $\eta_1/10$ from its tangent disk, $\Sigma_1$ and $\Sigma_2$ must intersect and, as
the gradient is less than $1/10$ and the angle between the disks is in $[\pi/3,2\pi/3]$,
the intersection is transverse.
\end{proof}

\begin{figure}[h]
\begin{center}
\resizebox{0.5\linewidth}{!}{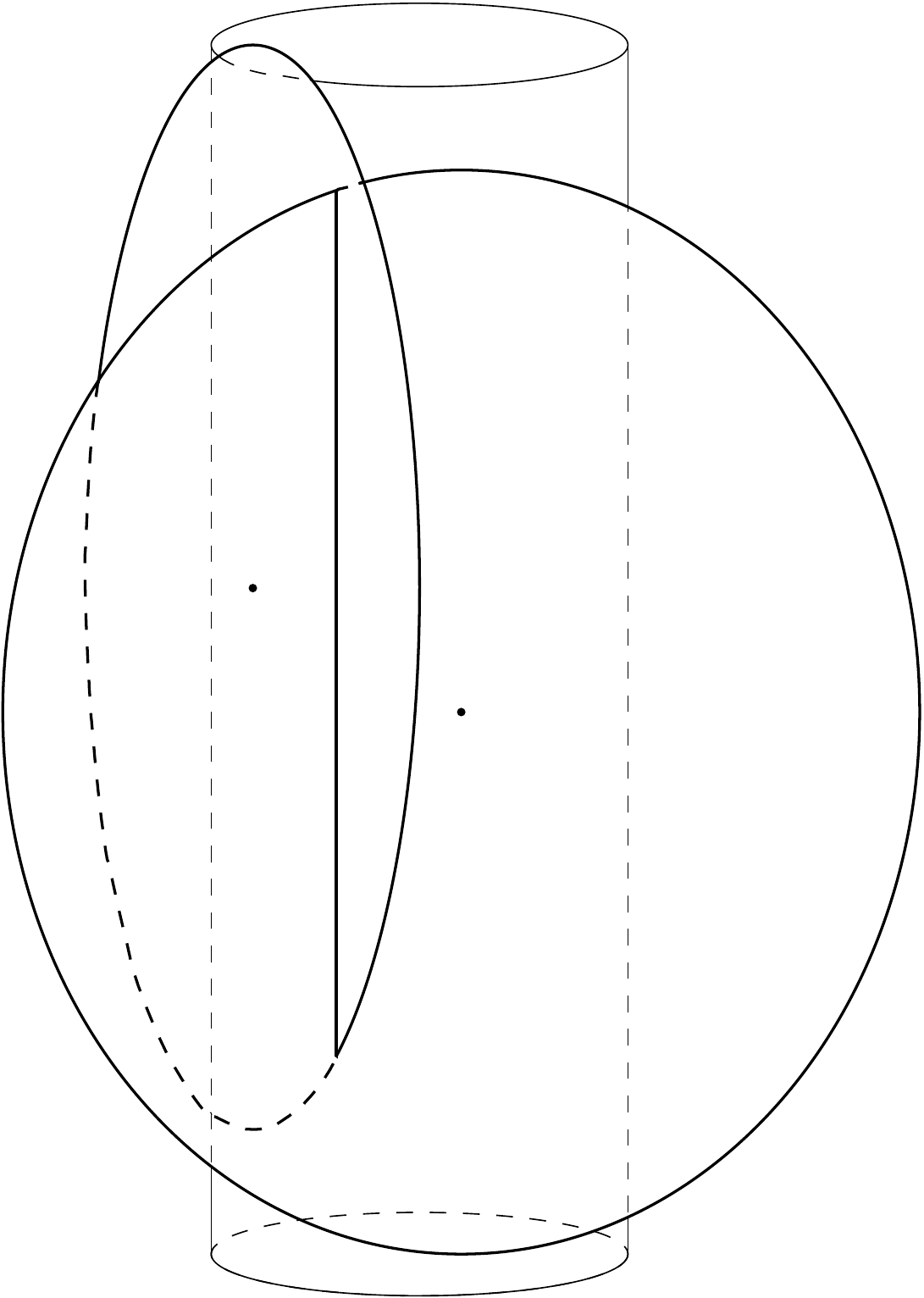}
\caption{\label{fig:fig1}}
\end{center}
\end{figure}

\subsection{The transversality result}

The main result of the section is then the following. We recall that $S_r=\partial N_r$.

\begin{prop}\label{prop:trans}
Let $\delta_0$, $k_0$ and $\eps_0$ be positive, then there is $\ell_0>0$ and $\barre R$ such that
the following is true. Let $\ell\le \ell_0$ and $N_{R_\ell}$ be the hyperbolic tubular neighborhood
of a geodesic loop $\gamma$ of length $\ell$ and such that the diameter of
$S_{R_\ell}$ is less than
$\delta_0$. Let $\Sigma$ be an embedded minimal surface in $N_{R_\ell}$ whose curvature is
bounded by $k_0$. Let $\bar r<R_\ell-\barre R$ and $p$ be a point in $\Sigma \cap S_{\bar
r}$ such that $d_\Sigma(p,\partial \Sigma)>\eps_0$. Then $\Sigma$ is not tangent to
$S_{\bar r}$ at $p$.
\end{prop}

We notice that for $\bar r=0$, $S_{\bar r}$ is just the central geodesic $\gamma$ so the
proposition states that $\Sigma$ can not be tangent to $\gamma$.

\begin{proof}
We start with some $\ell_0$ such that $R_\ell>10$. Let $r_0\le 1 $ and $\eta_0$ be given by
Lemma~\ref{lem:prox} for $k_0$ and $\eps_0$ (we assume $\eps_0\le 1$). We first prove that the result is true if
$\bar r\le r_0$.

Let $\Sigma$ be a minimal surface as in the statement of the proposition and assume that
$\Sigma$ is tangent at $p$ to $S_r$ for some $r$. We consider the lift 
$\widetilde\Sigma$ of $\Sigma$ to $\H^3$. $\widetilde \Sigma$ is then contained in a solid
cylinder $V_{R_\ell}$. 

The surface $\widetilde\Sigma$ is then an embedded minimal surface (may be non connected) which is invariant by
the action of the loxodromic transformation $\tau:(z,\theta,r)\mapsto (z+\ell,\theta+\alpha,r)$.
Let $p_1$ be a lift of $p$. We can assume that $p_1=(0,0,\bar r)$; if $\bar r=0$, we assume that
$\partial_r(0,0,0)$ is the unit normal vector to $\widetilde\Sigma$.

$S_{R_\ell}$ has a diameter less than $\delta_0$. So, for any $q$ in $B_{R_\ell}$, the intrinsic
disk of radius $\delta_0$ in
$B_{R_\ell}$ and center $q$ must contain an image of $(0,0,R_\ell)$ by some $\tau^n$.

Let us consider the domain $A_r=\{(z,\theta,r)\in B_r| z\in [-\frac{\delta_0}{\cosh
R_\ell},\frac{\delta_0}{\cosh R_\ell}], \theta\in[\frac\pi2-\frac{\delta_0}{\sinh
R_\ell},\frac\pi2+\frac{\delta_0}{\sinh R_\ell}]\}$, $A_{R_\ell}$ is a square in
$B_{R_\ell}$ whose
edges have length $2\delta_0$. 
So $A_{R_\ell}$ contains an image of $(0,R_\ell,0)$ by some $\tau^n$. This implies that $\tau^n$ is
the composition of a vertical translation by some $z'\in[-\frac{\delta_0}{\cosh
R_\ell},\frac{\delta_0}{\cosh R_\ell}]$ and a rotation by some $\theta'\in
[\frac\pi2-\frac{\delta_0}{\sinh R_\ell},\frac\pi2+\frac{\delta_0}{\sinh R_\ell}]$.

The point $p_2=\tau^n(p_1)=(z_2,\theta_2,\bar r)$ is another lift of $p$ in $A_{\bar r}$.
$\widetilde \Sigma$ is then also tangent to $B_{\bar r}$ at $p_2$. We have
$|\theta_2-\pi/2|\le \delta_0/\sinh R_\ell$ and $|z_2|\le\delta_0/\cosh R_\ell$. So we can choose
$\ell_0$ such that, for $\ell\le \ell_0$, $\delta_0/\sinh R_\ell\le \pi/6$ and $\delta_0/\cosh R_\ell\le
\eta_0$. Then we can apply Lemma~\ref{lem:prox} to the geodesic disks $\Sigma_i$ of radius
$\eps_0$ in $\widetilde\Sigma$ around $p_i$. Lemma~\ref{lem:prox} applies since, when
$\bar r =0$, the unit normal vector to $\Sigma_2$ at $p_2$ is
$\partial_r(z_2,\theta_2,0)$ with $|\theta_2-\pi/2|\le \delta_0/\sinh R_\ell$ ($\Sigma_2$ is the
image of $\Sigma_1$ by $\tau^n$). This gives that $\widetilde \Sigma$ has
transverse self-intersection which is impossible. So the result is proved for $\bar r\le
r_0$.

Let us now prove that we can extend this result to the region $r_0\le \bar r\le
R_L-\barre R$ for some $\barre R>0$.

If the result is not true, for any $n>0$, we can find a neighborhood $N_{R_{\ell_n}}$ of a closed
geodesic $\gamma_n$  of length $\ell_n\le \frac1n$ and a minimal surface
$\Sigma_n$ in $N_{R_{\ell_n}}$ which is tangent to $S_{r_n}$ at $p_n$ for some $r_n\le
R_{\ell_n}-\frac14\ln n$ (notice that $R_{\ell_n}-\frac14\ln n>0$).
Actually because of the first part we can assume $r_n>r_0$. In the following we
denote $R_{\ell_n}$ by $R_n$.

Let $\eta_1=\min(r_0/10,\eta_0)$ and replace the sequence $\Sigma_n$ by the sequence of
$\eta_1$-geodesic disks in $\Sigma_n$ centered at $p_n$. So we can be sure that $\Sigma_n$
never touches the central geodesic $\gamma_n$ and stays outside of $N_{r_0-\eta_1}$.

We lift $\Sigma_n$ to $M_{R_n}$ endowed with the metric~\eqref{eq:metric}. This
gives us a minimal surface $\widetilde\Sigma_n$ which is doubly periodic and may be non
connected. $\widetilde\Sigma_n$ is doubly periodic by translation in the $(z,\theta)$
parameters by two vectors $v_1^n,v_2^n$. Since $T_{R_n}$ has diameter less than $\delta_0$
we can choose $v_1^n,v_2^n$ of Euclidean length less than $\frac{\delta_0}{\sinh R_n}$. 

The point $p_n$ lifts to some point $\tilde p_n$ whose coordinates can be assumed to be
$(0,0,r_n)$ where $r_n\in(r_0,R_n-\frac12\ln n)$. We can assume that either $r_n$
converges to some $\bar r$ or to $\infty$. In the first case the ambient space around
$(0,0,\bar r)$ is $M_{\infty}=\R^2\times(0,+\infty)$ with the
metric~\eqref{eq:metric}. If $r_n\to
\infty$, we make the following change of coordinates $a=e^{r_n}z$, $b=e^{r_n}\theta$ and $\rho=r-r_n$. So the ambient space is now
$\R^2\times(r_0-\eta_1-r_n,R_n-r_n)$ with the metric
$$
\cosh^2(\rho+r_n)e^{-2r_n}da^2+\sinh^2(\rho+r_n)e^{-2r_n}db^2+d\rho^2
$$
As $n$ goes to $+\infty$, these metrics converge smoothly to $\frac{e^{2\rho}}4
(da^2+ db^2)+d\rho^2$ on $\R^3$. In this model, the vectors $v_1^n,
v_2^n$ become $e^{r_n}v_1^n$ and $e^{r_n}v_2^n$ whose lengths are less that $\frac{\delta_0
e^{r_n}}{\sinh R_n}=O(e^{r_n-R_n})=O(e^{-\frac12\ln n})\to 0$.

Actually, the cases $r_n\to\bar r$ and $r_n\to +\infty$ are very similar. Let us look first at
the case $r_n\to \bar r$. We notice that the metric satisfies the hypotheses of
Lemma~\ref{lem:estimgraph} (Appendix~\ref{sec:appendix}) for some parameter $A$ and for
$r\in[\bar r-\eta_1,R_n]$: we have $x_1=z$, $x_2=\theta$, $x_3=r$ and $h=\sinh$.
So there is a $C$ and a function $u_n$ defined on the Euclidean
disk $\{(z,\theta)\in \R^2|z^2+\theta^2\le 2C^2/\sinh^2 r_n\}$ such that $(z,\theta)\mapsto
(z,u_n(z,\theta),\theta)$ is a parametrization of a neighborhood of $\tilde p_n$ in
$\widetilde\Sigma_n$. Moreover we have $u_n(0,0)=r_n$, $\nabla u_n(0,0)=0$ and the estimates
$$
\|u_n-r_n\|\le A\eps_0\qquad\|\nabla u_n\|\le \sinh r_n \qquad \|\Hess u_n\|\le \frac1C\sinh^2 r_n.
$$
Here $\nabla$ denote the Euclidean gradient operator.

So the sequence $u_n$ is uniformly controlled in the $C^2$ topology and moreover $u_n$
solves the minimal surface equation~\eqref{eq:mse}. Thus, after considering a subsequence, $u_n$ converges
to some $u$ defined on $D_{\bar r}=\{(t,\theta)\in\R^2|t^2+\theta^2\le
\frac{2C^2}{\sinh^2\bar r}\}$ which solves the minimal surface equation.

If $r_n\to +\infty$, we apply the change of variables $a=e^{r_n}t$,
$b=e^{r_n}\theta$ and $\rho=r-r_n$. So we get a new function
$w_n(a,b)=u_n(e^{-r_n} a, e^{-r_n} b)-r_n$
defined on $\{(a,b)\in\R^2|a^2+b^2\le\frac{2C^2e^{2r_n}}{\sinh^2r_n}\}$. As above $w_n$
satisfies the estimates
\begin{align*}
\|w_n\|&\le A\eps_0&\|\nabla w_n\|&\le e^{-r_n}\sinh r_n & \|\Hess w_n\|&\le \frac1C e^{-2r_n}\sinh^2 r_n
\end{align*}
and solves a minimal surface equation~\eqref{eq:mse}. So we can assume it converges to some function $w$
defined on $\Delta=\{(a,b)\in\R^2|a^2+b^2\le 4C^2\}$.

Let us denote the surface $\{r=R\}$ by $P_R$. The surface $\widetilde\Sigma_n$ is doubly
periodic so it is tangent to $P_{r_n}$ at any point of the form $(0,0,r_n)+kv_1^n+lv_2^n$
for $(k,l)\in\Z^2$. Moreover, around these points, it is parametrized locally on
$D_{r_n}+kv_1^n+lv_2^n$ by $(z,\theta) \mapsto(z,u_{n,k,l}(t,\theta),\theta)$ where
$u_{n,k,l}(z,\theta)=u_n((z,\theta)-kv_1^n-lv_2^n)$. The surface $\widetilde \Sigma_n$ is
embedded, this implies that $u_n\le u_{n,k,l}$ or $u_n\ge u_{n,k,l}$ on $D_n\cap
(D_n+kv_1^n+lv_2^n)$ if it is non empty (notice that we can have $u_n\equiv u_{n,k,l}$ on
the intersection).

If $r_n\to \bar r$, let $v_0$ be a vector in $D_{\bar r}$. Since $v_i^n\to0$, there are
sequences $(k_n)_n$ and $(l_n)_n$ such that $k_nv_1^n+l_n v_2^n\to v_0$. As $n\to \infty$,
the sequence of functions $u_{n,k_n,l_n}$ then converges to $u_{v_0}$ on $D_{\bar r}+v_0$
where $u_{v_0}(\cdot)=u(\cdot-v_0)$. Because of $u_n\le
u_{n,k,l}$ or $u_n\ge u_{n,k,l}$, we get $u\le u_{v_0}$ or $u\ge u_{v_0}$ on $D_{\bar
r}\cap(D_{\bar r}+v_0)$.

If $r_n\to \infty$, we can do the same with the change of coordinates since
$e^{r_n}v_i^n\to 0$. So for any $v_0\in\Delta$, we have $w\le w_{v_0}$ or
$w\ge w_{v_0}$ on $\Delta\cap (\Delta+v_0)$ where
$w_{v_0}(\cdot)=w(\cdot+v_0)$.

We now consider the case $r_n\to \bar r$ (the second one is similar). Let $G$ be the
totally geodesic surface in $M_\infty$ tangent to $P_{\bar r}$ at $(0,\bar r,0)$. As
$\widetilde\Sigma$, $G$ can be described as the graph of a function $h$ over $D_{\bar r}$.
We have $h(0)=\bar r$ and there is some $\alpha>0$ such that, over $D_{\bar r}$,
$h(z,\theta)\ge \bar r+\alpha(z^2+\theta^2)$. This second property comes from the fact that
the principal curvatures of $P_{\bar r}$ with respect to $\partial_r$ are $-\tanh \bar r<0$
and $-\coth \bar r<0$. The functions $u$ and $h$ are two solutions of the minimal surface
equation~\eqref{eq:mse} with the same value and the same gradient at the origin. So by Bers theorem, the function $u-h$
looks like a harmonic polynomial of degree at least $2$.

If the degree of the polynomial is $2$, on can find $v_0\in D_{\bar r}\setminus\{(0,0)\}$ such that
$(u-h)(v_0)>0$ and $(u-h)(-v_0)>0$. Then we have
\begin{gather*}
u(v_0)>h(v_0)>h(0,0)=u(0,0)=u_{v_0}(v_0)\\
u_{v_0}(0,0))=u(-v_0)>h(-v_0)>h(0,0)=u(0,0)
\end{gather*}
So this contradicts $u\le u_{v_0}$ or $u\ge u_{v_0}$ on the whole $D_{\bar r}\cap (D_{\bar
r}+v_0)$

If the degree is at least $3$, the growth at the origin of $h$ implies that $u\ge \bar r$
on a smaller disk $D'\subset D_{\bar r}$ and $u>\bar r$ on $D'\setminus\{(0,0)\}$. So if
$v_0\in D'\setminus\{(0,0)\}$ we have
$$
u(v_0)>\bar r=u_{v_0}(v_0) \text{ and }u_{v_0}(0,0)=u(-v_0)>\bar r=f(0,0)
$$
Once again, this contradicts $u\le u_{v_0}$ or $u\ge u_{v_0}$ on the whole $D_{\bar r}\cap
(D_{\bar r}+v_0)$

If $r_n\to \infty$, the same argument can be done with a totally geodesic surface tangent to
the horosphere.
\end{proof}

\subsection{A first area estimate}

The preceding result allows us to estimate the area of a minimal surface with bounded curvature in a tubular end.

\begin{cor}\label{cor:estimarea1}
Let $\delta_0$ and $k_0$ be positive, then there is $\ell_0$ and $\barre R$ such
that the following is
true. Let $\ell\le \ell_0$ and $N_{R_\ell}$ be the hyperbolic tubular neighborhood
of a geodesic loop $\gamma$ of length $\ell$ and such that the diameter of
$S_{R_\ell}$ is less than
$\delta_0$. Let $0<R\le R_\ell-\barre R$ and $\Sigma$ be a compact embedded
minimal surface in $N_{R+1}$
whose curvature is bounded by $k_0$ and $\partial\Sigma\subset S_{R+1}$. Then one of the
following possibilities occurs 
\begin{enumerate}
\item $\Sigma\cap N_R=\emptyset$
\item $\Sigma\cap N_R$ is a finite union of minimal disks. Each of these disks
has boundary curve  homotopic to a parallel of $S_R=\partial N_R$ and $|\Sigma\cap
N_R|\ge 2\pi(\cosh R-1)$.
\end{enumerate}

\end{cor}

A parallel of $S_R$ is a curve $\{z=const.\}$ in the tubular coordinates.

\begin{proof}
Let $\ell_0$ and $\barre R$ be given by Proposition~\ref{prop:trans} for $\delta_0$, $k_0$
and $\eps_0=1$.
Let $\Sigma$ be as in the statement of the corollary and assume $\Sigma\cap
N_R\neq\emptyset$. By Proposition~\ref{prop:trans}, $\Sigma$ is transverse to
the foliation $(S_r)_r$ of $N_R$. So any connected component of $\Sigma\cap
N_R$ intersects the geodesic loop $\gamma$ transversely. This implies that in
$N_\eps$ for $\eps$ small each connected component of $\Sigma\cap N_{\eps}$ is a disk
whose boundary is homotopic to a parallel. Thus this description extends by
transversality to $\Sigma\cap N_R$. Let $\Pi$ be the geodesic projection from
$N_R$ to a geodesic parallel disk $\Delta$ (\textit{i.e.} the map $(z,\theta,r)\mapsto
(z_0,\theta,r)$ for some $z_0$). This map is a contraction mapping and it is surjective on
any disk component of $\Sigma\cap N_R$ since the boundary of such a disk is homotopic to a
parallel. As a consequence the area of such a disk component is at least that of $\Delta$,
\textit{i.e.} $2\pi(\cosh(R)-1)$.
\end{proof}


\section{A maximum principle}
\label{sec:maxprinciples}

One aim of this section is to study some aspect of the behavior of minimal surfaces
in a tubular end. Actually we need to study this in a more general setting. So we consider
the ambient space $C=T\times [a,b]$ endowed with some reference metric $\barre
g=h^2(x_3)d\bar\sigma^2+dx_3^2$ where $d\bar \sigma^2$ is a 
flat metric on the torus $T$.
We consider orthonormal coordinates $(x_1,x_2)$ on $T$ associated to $d\bar
\sigma^2$; so $\bar g=h^2(x_3)(dx_1^2+dx_2^2)+dx_3^2$.

On $C$, we also consider a second metric 
$g=a_{kl}(x_1,x_2,x_3)dx_kdx_l$. For $s\in[a,b]$,
we denote $C_s=T\times[s,b]$
and $T_s=T\times\{s\}$. We are going to make several hypotheses on the metrics
$\bar g$ and $g$. In order to formulate them, we need the following notation:
for $k_1,k_2,k_3,k_4,k_5\in \{1,2,3\}$ and $p\le 5$, we define 
$$
n_p(k_1,\dots,k_p)=\#\{i\in\{1,\dots,p\}|k_i\in\{1,2\}\}.
$$ 
The hypotheses on $\bar g$ and $g$ are: there is $A\ge 1$ such that 
\begin{itemize}
\item[H1] $\frac1{A^2}\bar g\le g\le A^2\bar g$
\item[H2] $\frac{|h'|}h\le A$, $\frac{|h''|}h\le A$ and $\frac{|h'''|}{h}\le A$.
\item[H3] $|a_{kl}|\le A h^{n_2(k,l)}(x_3)$, $|\partial_ia_{kl}|\le A 
h^{n_3(k,l,i)}(x_3)$,
$|\partial_i\partial_j a_{kl}|\le A h^{n_4(k,l,i,j)}(x_3)$ and $|\partial_i\partial_j\partial_m 
a_{kl}|\le 
A h^{n_5(k,l,i,j,m)}(x_3)$.
\item[H4] $h'\le 0$ and the mean curvature vector of $T_s$ with respect to $g$ points in the
$\partial_{x_3}$ direction (this is also true for the metric $\bar g$ since $h'\le 0$)
\end{itemize}

One consequence of H1 and H2 is that the sectional curvatures of $\bar g$ are uniformly
bounded. Actually by $H1$ and $H3$ the sectional curvatures of $g$ are also uniformly
bounded. We also notice that these hypotheses does not depend on the choice of the 
orthonormal coordinates on $(T,d\bar \sigma^2)$.

\subsection{The maximum principle}
We have the following maximum principle for embedded minimal surfaces in $C$ endowed 
with the metric $g$.
\begin{prop}\label{prop:max3}
Let $i_0\in\N$, then there is $h_0$ such the following is true. Assume that 
$h(a)\le h_0$ and let $\Sigma$ be an embedded minimal surface in 
$(C,g)$ whose non empty boundary is inside $T_a$ and its index is less than $i_0$.  Then 
$\Sigma\cap C_{a+1/2}=\emptyset$.
\end{prop}

We notice that $h_0$ will depend on $i_0$, $A$ and the metric $d\bar \sigma^2$. We 
also notice that this control on $h$ is actually a control on the size of the torus $T_a$.

\begin{proof}
If the proposition is not true there is a sequence of function $h_n$ with $h_n(a)\to 0$ 
and a minimal surface 
in $S_n\in (C,g_n)$ ($g_n=a_{n,kl}(x_1,x_2,x_3)dx_kdx_l$)
such that $\partial S_n\subset T_a$, its index is less than $i_0$ and $\Sigma\cap 
C_{a+1/2}\neq \emptyset$.

Let $s_n$ be the maximum of the $x_3$ coordinate on $S_n$, $x_3\ge a+1/2$.
Let us define $\lambda_n=(h_n(s_n))^{-1}$. Then we change the  coordinates by  
$y_1=x_1$, $y_2=x_2$ and $y_3=\lambda_n(x_3-s_n)$ and blow up the metric by a factor 
$\lambda_n$. 
This gives us a minimal surface $\Sigma_n$ in 
$T\times[\lambda_n(a-s_n),0]$ that touches $T_0$ and with boundary in 
$T_{\lambda_n(a-s_n)}$ (we notice that $\lambda_n(a-s_n)\to -\infty$). The ambient 
metric is then $\tilde g_n=b_{n,kl}(y_1,y_2,y_3)dy_kdy_l$ where 
\[
b_{n,kl}(y_1,y_2,y_3)=a_{n,kl}(y_1,y_2,y_3/\lambda_n+s_n)\lambda_n^{n_2(k,l)}
\]
The reference metric becomes 
\[
\frac{h_n^2(y_3/\lambda_n+s_n)}{h_n^2(s_n)}(dy_1^2+dy_2^2)+dy_3^2
\]
Because of the hypothesis $H2$, considering a subsequence, this metric converges to 
the flat metric $d\bar\sigma^2+dy_3^2$ in $C^{2,\alpha}$ topology. Because of H3, 
considering a new subsequence, the metrics $\tilde g_n$  converges to a flat metric 
$\bar h=\bar b_{kl}dy_kdy_l$ in $C^{2,\alpha}$ topology. For example we have
\[
\partial_ib_{n,kl}=\partial_i a_{n,kl}(y_1,y_2,y_3/\lambda_n+s_n)\lambda_n^{n_{k,l,i}-1}
\]
So
\[
|\partial_ib_{n,kl}|\le A\frac{h_n^{n_3(k,l,i)}(y_3/\lambda_n+s_n)} 
{h_n^{n_3(k,l,i)}(s_n)}h_n(s_n)\to 0
\]

Once this is known, the arguments in order to conclude use the fact that $\Sigma_n$ 
converges to a minimal lamination in $T^2\times \R_-$ endowed with the flat metric $\bar 
h$ : the precise argument can be found in the proof of Proposition~1 in 
\cite{CoHaMaRo2}.
\end{proof}
\begin{remarq}
We notice that $h_0$ can be chosen uniformly if $d\bar\sigma^2$ lies in a compact 
subset of flat metrics on $T$. 
\end{remarq}


\subsection{Some applications}

In this section, we will see some consequences of the above result.

The case of cusp ends $E=T\times\R_+$ endowed with $\bar g=e^{-2x_3} d\bar \sigma^2+dx_3^2$ is the
simplest one. Indeed in this case the metric $g$ is the reference metric $\bar g$.
Let $L$ be fixed and replace $x_3\mapsto e^{-x_3}$ by some non increasing function $h$
such that $h(x_3)=e^{-x_3}$ on $[0,L]$ and $h(x_3)= e^{-(L+1/2)}$ for $x_3>L+1$. Then all the
above hypotheses $H_2,H_4$ are satisfied (the constant $A$ can be chosen independently of $L$). So
Proposition~\ref{prop:max3} yields: if $\partial E$ has small diameter, then no compact
embedded minimal surface with boundary inside $\partial E$  and index less 
than $1$ can enter in $E_{1/2}$. As a
consequence, in a cusp manifold $M$, there is $\eps>0$ such that any compact embedded
minimal surface in $M$ less than $1$ is contained in $M_{[\eps,\infty)}$.

The second case of interest concerns the tubular ends. 

\begin{prop}\label{prop:estimarea}
Let $K$ be a compact set of flat tori $T$. Then there are $\ell_0$ and $\barre R$ such 
the following is
true. Let $\ell\le \ell_0$ and $N_{R_\ell}$ be a hyperbolic tubular neighborhood of a
geodesic loop of length $\ell$ such that $S_{R_\ell}$ belongs in $K$. Let $0<R\le 
R_\ell-\barre R$ and $\Sigma$ be a compact embedded minimal
surface in $N_{R+1}$ with $\partial\Sigma\in S_{R+1}$ with index less than $1$. Then one 
of the following possibilities occurs
\begin{enumerate}
\item $\Sigma\cap N_R=\emptyset$
\item $\Sigma\cap N_1\neq\emptyset$
\end{enumerate}
Moreover there is a universal constant $\kappa$ such that, in the second case and for any
$3\le R\le R_\ell-\barre R$,
$|\Sigma\cap N_R|\ge \kappa \s_0 e^{R-R_\ell}$ where $\s_0\le 1$ is a lower bound on the
systole of $T_{R_\ell}$.
\end{prop}

\begin{proof}
We first prove that $\Sigma\cap N_R=\emptyset$ or $\Sigma\cap N_1\neq \emptyset$. We have
seen in Section~\ref{sec:cutuend} that we can consider, on $N_{R_\ell}$, a coordinate system $C=T\times [0,R_\ell)$
endowed with the metric $g=\sinh^2(R_\ell-x_3)dx_1^2+\cosh^2(R_\ell-x_3)dx_2^2+dx_3^2$
($(x_1,x_2)$ are orthonormal coordinates on $T$).

In order to fit with the notations of the preceding section we should introduce the 
coordinates $y_i=\sinh(R_\ell)x_i$ $i=1,2$ and $y_3=x_3$. The first two are orthonormal 
coordinates on  $(T,d\bar\sigma^2)=\sinh(R_\ell)^2(dx_1^2+dx_2^2)$. So we define 
$\bar 
g 
=h^2(x_3)(dy_1^2+dy_2^2)+dy_3^2$ with $h(x)=\sinh(R_\ell-x)/\sinh(R_\ell)$ and in these 
coordinates the metric $g$ can be written
$$
\frac1{\sinh^2(R_\ell)}(\sinh^2(R_\ell-y_3)dy_1^2+\cosh^2(R_\ell-y_3)dy_2^2)+dy_3^2
$$
 
There is 
$A>0$ (that does not depend on $\ell$) such that
$g$ and $\bar g$ satisfy hypotheses H1, H2, H3 and H4 on $T\times[0,R_\ell-\frac12)$. 
Moreover we notice that, since $S_{R_\ell}$ belongs to $K$,  $(T,d\bar\sigma^2)$ 
belongs to a compact set of flat tori. 


Let $h_0$ be given by Proposition~\ref{prop:max3} and let $\barre R$
be such that $\sinh(R_\ell-\barre R)\le h_0\sinh(R_\ell)$. Consider $0<R\le R_\ell-\barre 
R$
and let $\Sigma$ be an embedded minimal surface in $N_{R+1}\setminus N_{\frac12}$ 
with $\partial\Sigma\in 
S_{R+1}$ and index less than $1$.
If $\Sigma\cap N_1=\emptyset$, $\Sigma$ can be seen as a minimal surface in $(C, g)$
with boundary in $T_s$ where $s=R_\ell-(R+1)$. Since $h(s)\le h_0$, 
Proposition~\ref{prop:max3} gives $\Sigma\cap
C_{s+1/2}=\emptyset$. So in the tubular coordinates, we have $\Sigma\cap N_R=\emptyset$.

In the second case we now prove the area estimate. For this we use the tubular
coordinates. We notice that $\Sigma$ must meet all the tori $S_r$ for $1\le r\le R+1$.

Since $g\le \cosh^2 r(dz^2+d\theta^2)+dr^2$ and the systole of $T_{R_L}$ is at least
$\s_0$, the disk $\{(z+z_0,\theta+\theta_0,R_\ell) |\,z^2+\theta^2\le \frac{\s_0^2}{4\cosh^2
R_\ell}\}$ is embedded in $S_{R_\ell}$ for any
$t_0,\theta_0$. For any $\rho\in [3/2,R_\ell]$, let us define $a=\frac{\sinh \rho}{\cosh R_\ell}\frac{\s_0}4\le
\frac14$. The cylinder $Y_\rho=\{(z+z_0,\theta+\theta_0,r)| r\in [\rho-2a,\rho]\textrm{ and
}z^2+\theta^2\le \frac{\s_0^2}{4\cosh^2 R_\ell}\}$ is embedded in $N_\rho$. $Y_\rho$
contains the geodesic ball of center $(z_0,\theta_0,\rho-a)$ and radius $a$ which is then
embedded in $N_\rho$. Indeed, in the cylinder, we have
$$
g\ge \sinh^2(\rho-2a)(dz^2+d\theta^2)+dr^2\ge \frac14\sinh^2 \rho(dz^2+d\theta^2)+dr^2
$$
So the geodesic ball is contained in $\{(z+z_0,\theta+\theta_0,r+\rho-a)| \frac14\sinh^2
\rho(z^2+\theta^2)+r^2\le a^2\}$ which is a subset of $Y_\rho$.

Since $\Sigma$ meets any $S_r$ for $r\ge 1$, for any $\rho$ we can select $z_0,\theta_0$
such that $(z_0,\theta_0,\rho-a)\subset \Sigma$. So by the monotonicity formula in $\H^3$,
$|\Sigma\cap Y_\rho|\ge \pi a^2$. We are going to sum over all these contributions to
estimate the area of $\Sigma$.

Let $c(s)=(z(s),\theta(s),R_\ell)$ be a parametrization of a systole of $S_{R_\ell}$ and
consider the surface $S$ in $N_{R_\ell}$ parametrized by $X:(s,r)\in \S^1\times[1,R_\ell]\mapsto
(z(s),\theta(s),r)$. So, for $\rho_1<\rho_2$, we can estimate
\begin{align*}
|S\cap(N_{\rho_2}\setminus N_{\rho_1})|&\le
\int_{\rho_1}^{\rho_2}\int_{\S^1}(\cosh^2(r)z'^2(s)+\sinh^2(r)\theta'^2(s))^{1/2}drds\\
&\le\int_{\rho_1}^{\rho_2}\int_{\S^1}\frac{\cosh r}{\sinh R_\ell}(\cosh^2 (R_\ell)
z'^2(s)+\sinh^2(R_\ell)\theta'(s))^{1/2}drds\\
&\le \frac{\s_0}{\sinh R_\ell}(\sinh \rho_2-\sinh\rho_1)\\
&\le \frac{2\s_0}{\sinh R_\ell}\cosh\frac{\rho_1+\rho_2}2\sinh\frac{\rho_2-\rho_1}2
\end{align*}
So in $N_\rho\setminus N_{\rho-2a}$
\begin{align*}
|S\cap (N_\rho\setminus N_{\rho-2a})|&\le \frac{2\s_0}{\sinh R_\ell}\cosh(\rho-a)\sinh a\\
&\le \frac{2\kappa \s_0}{\sinh R_\ell}\cosh(\rho) a\\
&\le \frac{2\kappa' \s_0}{\cosh R_\ell}\sinh(\rho)a\\
&\le 8\kappa'a^2\le \frac{8\kappa'}\pi|\Sigma\cap Y_\rho|\le \frac{8\kappa'}\pi|\Sigma\cap (N_\rho\setminus N_{\rho-2a})|
\end{align*}
for some universal constant $\kappa$ and $\kappa'$.

So considering a disjoint union of $N_\rho\setminus N_{\rho-2a}$ in $N_R\setminus N_1$ that
covers $N_R\setminus N_{3/2}$, we obtain
\begin{align*}
|\Sigma \cap N_R|&\ge \frac\pi{8\kappa'}|S\cap (N_R\setminus N_{3/2})|\\
&\ge\frac\pi{8\kappa'}\int_{3/2}^R\int_{\S^1}(\cosh^2(r)y'^2(s)+\sinh^2(r)\theta'^2(s))^{1/2}drds\\
&\ge\frac\pi{8\kappa''} \frac{\s_0}{\cosh R_\ell}(\cosh R-\cosh 3/2)\ge \kappa''' \s_0
e^{R-R_\ell}\\
\end{align*}
for any $R\ge 3$ and some universal constant $\kappa'''$.
\end{proof}


\section{The min-max theory}\label{sec:minmax2}

In this section we recall some definitions and results of the min-max theory for minimal
surfaces. There are basically two settings for this theory: the discrete and the
continuous one. We recall the main points that we will use in the sequel.


\subsection{The discrete setting}
The discrete setting for the min-max theory was developed by Almgren and Pitts (see
\cite{Alm,Pit}).

Let $M$ be a compact orientable $3$-manifold with no boundary. The Almgren-Pitts min-max theory
deals with discrete families of elements in $\boZ_2(M)$ \textit{i.e.} integral rectifiable
$2$-currents in $M$ with no boundary.

If $I=[0,1]$, we define some cell complex structure on $I$ and $I^2$.
\begin{defn}
Let $j$ be an integer. $I(1,j)$ is the cell complex on $I$ whose $0$-cells are points
$[\frac i{3^j}]$ and $1$-cells are intervals $[\frac i{3^j},\frac{i+1}{3^j}]$.

The cell complex $I(2,j)$ on $I^2$ is $I(2,j)=I(1,j)\otimes I(1,j)$.
\end{defn}

For these cell complexes we can associate some notations
\begin{itemize}
\item $I(m,j)_0$ denotes the set of $0$-cells of $I(m,j)$.
\item $I_0(1,j)$ denotes the set of $0$-cells ${[0],[1]}$.
\item The distance between two elements of $I(m,j)_0$ is 
$$
\bfd: I(m,j)_0\times I(m,j)_0\to \N\ ;\ (x,y)\mapsto 3^j\sum_{i=1}^m|x_i-y_i|
$$
\item The projection map $n(i,j) : I(m,i)_0\to I(m,j)_0$ is defined by $n(i,j)(x)$
is the unique element in $I(m,j)_0$ such that 
$$\bfd(x,n(i,j)(x))=\inf\{\bfd(x,y), y\in I(1,j)_0\}.$$
\end{itemize}

Let $\phi:I(m,j)_0\to \boZ_2(M)$ be a map. The fineness of $\phi$ is defined by
$$
\bff(\phi)=\sup\left\{\frac{\MM(\phi(x)-\phi(y))}{\bfd(x,y)},x,y\in I(m,j)_0 \text{ and
}x\neq y\right\}
$$
where $\MM$ is the mass of a current.

We write $\phi :I(1,j)_0\to (\boZ_2(M),\{0\})$ to mean $\phi(I(1,j)_0)\subset
\boZ_2(M)$ and $\phi(I_0(1,j))=\{0\}$.

\begin{defn}
Let $\delta$ be positive and $\phi_i: I(1,k_i)_0\to (\boZ_2(M),\{0\})$ for $i=1,2$.
$\phi_1$ and $\phi_2$ are $1$-homotopic in $(\boZ_2(M),\{0\})$ with fineness $\delta$ if
there is $k_3\in \N$, $\max(k_1,k_2)\le k_3$ and a map
$$
\psi : I(2,k_3)_0\to \boZ_2(M)
$$
such that 
\begin{itemize}
\item $f(\psi)\le \delta$;
\item $\psi([i-1],x)=\phi_i(n(k_3,k_i)(x))$ for $x\in I(1,k_3)_0$;
\item $\psi(I(1,k_3)_0\times\{[0],[1]\})=0$.
\end{itemize}
\end{defn}

The main objects in the discrete min-max theory are the $(1,\MM)$-homotopy sequences.

\begin{defn}
A $(1,\MM)$-homotopy sequence of maps into $(\boZ_2(M),\{0\})$ is a sequence of maps $\{\phi_i\}_{i\in\N}$,
$$
\phi_i:I(1,k_i)_0\to (\boZ_2(M),\{0\}),
$$
such that $\phi_i$ is $1$-homotopic to $\phi_{i+1}$ in $(\boZ_2(M),\{0\})$ with fineness
$\delta_i$ and
\begin{itemize}
\item $\lim_{i\to \infty}\delta_i=0$;
\item $\sup_i\{\MM(\phi_i(x)), x\in I(1,k_i)_0\}<+\infty$.
\end{itemize}
\end{defn}

Moreover we have a notion of discrete homotopy between $(1,\MM)$-homotopy sequences
\begin{defn}
Let $S_j=\{\phi_i^j\}_{i\in\N}$ ($j=1,2$) be two $(1,\MM)$-homotopy sequences of maps into
$(\boZ_2(M),\{0\})$. $S_1$ is homotopic to $S_2$ if there is a sequence
$\{\delta_i\}_{i\in\N}$ such that
\begin{itemize}
\item $\lim \delta_i=0$;
\item $\phi_i^1$ is $1$-homotopic to $\phi_i^2$ in $(\boZ_2(M),0)$ with fineness $\delta_i$.
\end{itemize}
\end{defn}

This notion defines an equivalence relation between $(1,\MM)$-homotopy sequences. The set
of equivalence classes is denoted by $\pi_1^\#(\boZ_2(M),\MM,\{0\})$. The Almgren-Pitts
theory says that $\pi_1^\#(\boZ_2(M),\MM,\{0\})$ is naturally isomorphic to the homology
group $H_3(M,\Z)$ (see Theorem 4.6 in \cite{Pit} and \cite{Alm}). We denote by $\Pi_M$ the
element of $\pi_1^\#(\boZ_2(M),\MM,\{0\})$ that corresponds to the fundamental class in
$H_3(M)$. If $S\in \Pi_M$ we say that $S$ is a discrete sweep-out of $M$.

For $S=\{\phi_i\}_i$ a $(1,\MM)$-homotopy sequence we define 
$$
\LL(S)=\limsup_{i\to\infty}\max\{\MM(\phi_i(x)),x\in I(1,k_i)_0\}
$$
If $\Pi\in \pi_1^\#(\boZ_2(M),\MM,\{0\})$ is an equivalence class, we define the width associated to $\Pi$ by 
$$
W(\Pi)=\inf\{\LL(S),S\in \Pi\}
$$
For $\Pi=\Pi_M$, the number $W_M=W(\Pi_M)$ is call the width of the manifold $M$. The
Almgren-Pitts theory says that this number is positive and is $\LL(S)$ for some particular
$S\in \Pi_M$. If $S=\{\phi_i\}_i$ we say that ${\phi_{i_j}(x_j)}_j$ ($x_j\in
I(1,k_{i_j})$) is a min-max sequence if $\MM(\phi_{i_j}(x_j))\to W_M$. 

\begin{thm}[Pitts \cite{Pit}]\label{th:minmax}
Let $M$ be a closed $3$-manifold, then there is $S=\{\phi_i\}_i\in \Pi_M$ with
$\LL(S)=W_M$ and a min-max sequence $\{\phi_{i_j}(x_j)\}_j$ that converges (in the varifold
sense) to an integral varifold whose support is a finite collection of embedded connected
disjoint minimal surfaces $\{S_i\}_i$ of $M$. So there are positive numbers $\{n_i\}_i$
such that 
$$
W_M=\sum_{i=1}^pn_i|S_i|
$$
\end{thm}
A consequence of this result is that there is always a minimal surface $S$ in $M$ such
that $|S|\le W_M$. Actually, Zhou~\cite{Zho} proved that, if
$S_i$ is a non orientable minimal surface produced by the above theorem, then $n_i$ is
even.


\subsection{The continuous setting}

The continuous setting was developed by Colding and De Lellis \cite{CoDeL}. Here we present it
with the modifications made by Song in~\cite{Son}.

Let $M$ be Riemannian $3$-manifold and consider $N\subset M$ a bounded open subset
whose boundary $\partial N$, when non empty, is a rectifiable surface of finite
$\boH^2$-measure. Moreover when $\partial N\neq\emptyset$, we assume that each connected component $C$
of
$\partial N$ separates $M$.

If $a<b\in \R$, we then have the following definitions.

\begin{defn}
A family of $\boH^2$-measurable closed subsets $\{\Gamma_t\}_{t\in[a,b]}$ in $N\cup\partial
N$ with finite $\boH^2$-measure is called a generalized smooth family if
\begin{itemize}
\item for each $t$ there is a finite set $P_t\in N$ such that $\Gamma_t\cap N$ is a smooth
surface in $N\setminus P_t$ or the empty set;
\item $\boH^2(\Gamma_t)$ depends continuously in $t$ and $t\mapsto \Gamma_t$ is continuous
in the Hausdorff sense;
\item on any $U\subset\subset N\setminus P_{t_0}$, $\Gamma_t\xrightarrow[]{t\to t_0}
\Gamma_{t_0}$ smoothly in $U$.
\end{itemize}
\end{defn}

We notice the smoothness hypothesis is only made on $\Gamma_t\cap N$ so this allows
$\partial N$ to be non smooth. We now define the notion of continuous sweep-out in this
setting.

\begin{defn}
If $\partial N=\emptyset$, a generalized smooth family $\{\Gamma_t\}_{t\in[a,b]}$ is
called a continuous sweep-out of $N$ if there exists a family of open subsets
$\{\Ome_t\}_{t\in[a,b]}$ of $N$ such that 
\begin{itemize}
\item[(sw1)] $(\Gamma_t\setminus \partial\Ome_t)\subset P_t$ for any $t$;
\item[(sw2)] $\boH^3(\Ome_t\triangle \Ome_s)\to 0$ as $t\to s$ (where $\triangle$
denotes the symmetric difference of subsets).
\item[(sw3)] $\Ome_a=\emptyset$ and $\Ome_b=N$;
\end{itemize}

If $\partial N\neq \emptyset$, for a open subset $\Ome\subset N$ we denote
$\partial_*\Ome=\partial\Ome\cap N$. A continuous sweep-out of $N$ is then required to satisfy
(sw1) and (sw2) above except that $\partial$ is replaced by $\partial_*$ and $t>a$ in {(sw1)}. Moreover (sw3) is replaced by
\begin{itemize}
\item[(sw3')] $\Ome_a=\emptyset$, $\Ome_b=N$, $\Sigma_a=\partial N$ and $\Sigma_t\subset
N$ for $t>a$.
\end{itemize}
\end{defn}
For a continous sweep-out as above $\{\Gamma_t\}_{t\in[a,b]}$, we define the quantity
$\LL(\{\Gamma_t\})=\max_{t\in[a,b]}\boH^2(\Gamma_t)$. When $\partial N$ is a smooth
surface, constructing a continuous sweep-out can be done in the following way. Let $f:N\to
[0,1]$ be a Morse function such that $\{f^{-1}(0)\}=\partial N$. Then if
$\Gamma_t=f^{-1}(t)$ for $t\in [0,1]$, $\{\Gamma_t\}_{t\in[0,1]}$ is a sweep-out of $N$.

Two continuous sweep-outs $\{\Gamma_t^1\}_{t\in[a,b]}$ and $\{\Gamma_t^2\}_{t\in[a,b]}$ are said to
be homotopic if, informally, they can be continuously deformed one to the other (the precise
definition is Definition~8 in~\cite{Son}). Then a family $\Lambda$ of sweep-outs is called
homotopically closed if it contains the homotopy class of each of its elements. For such a
family $\Lambda$, we can define the width
associated to $\Lambda$ as
$$
W(N,\partial N,\Lambda)= \inf_{\{\Gamma_t\}\in\Lambda} \LL(\{\Gamma_t\})
$$

As in the discrete setting this number can be realized as the mass of some varifold supported by
smooth disjoint minimal surfaces (see Theorem~12 in~\cite{Son}).


\subsection{From continous to discrete}
In order to construct discrete sweep-outs of a closed orientable $3$-manifold $M$, we will use a result obtained by
Zhou (see Theorem~5.1 in \cite{Zho2}). We denote by $\boC(M)$ the space of subsets in $M$
with finite perimeter. Let $\boF$ denote the flat metric on $\boZ_2(M)$.

\begin{thm}\label{th:conttodis}
Let $\Phi:[0,1]\to (\boZ_2(M),\boF)$ be a continuous map such that 
\begin{itemize}
\item $\Phi(0)=0=\Phi(1)$;
\item $\Phi(x)=\partial[[\Ome_x]]$, $\Ome_x\in \boC(M)$ for all $x\in[0,1]$;
\item $M_0=\emptyset$ and $M_1=M$
\item $\sup_x\MM(\Phi(x))<+\infty$.
\end{itemize}
Then there exists a discrete sweep-out $S$ such that
$$
\LL(S)\le\sup_{x\in[0,1]} \MM(\Phi(x))
$$
\end{thm}

Here $[[\Ome]]$ denotes the element of $\boZ_3(M)$ corresponding to $\Ome$.

Let us notice that if $\{\Gamma_t\}_{t\in[0,1]}$ is a continuous sweep-out of a compact orientable
Riemannian $3$-manifold then $\Phi:t\mapsto [[\Gamma_t]]\in \boZ_2(M)$ satisfies the
hypotheses of the above theorem (as above $[[\Gamma_t]]$ denotes the element of
$\boZ_2(M)$ corresponding to $\Gamma_t$).


\section{The quantity $\boA_1(M)$}\label{sec:minmax}

In this section we recall some results the authors obtained in \cite{MaRo2} (see also 
\cite{CoHaMaRo2}).


\subsection{The quantity $\boA_1(M)$ for compact $M$}

If $M$ is a closed orientable Riemannian $3$-manifold, we denote by $\boO$ the collection of all smooth
orientable embedded closed minimal surfaces in $M$ and $\boU$ the collection of all
smooth non-orientable ones. We then define
$$
\boA_1(M)=\inf(\{|\Sigma|, \Sigma\in \boO\}\cup \{2|\Sigma|, \Sigma\in \boU\})
$$
One of the results of \cite{MaRo2} is the following theorem (Theorem B in \cite{MaRo2})
\begin{thm}\label{th:leastarea}
Let $M$ be an oriented closed Riemannian $3$-manifold. Then $\boA_1(M)$ is equal to
one of the following possibilities.
\begin{enumerate}
\item $|\Sigma|$ where $\Sigma\in \boO$ is a min-max surface of $M$ associated to the
fundamental class of $H_{3}(M)$, $\Sigma$ has index $1$, is separating and $\boA_1(M)=W_M$.
\item $|\Sigma|$ where $\Sigma\in\boO$ is stable.
\item $2|\Sigma|$ where $\Sigma\in \boU$ is stable and its orientable $2$-sheeted
cover has index $0$.
\end{enumerate}

Moreover, if $\Sigma\in\boO$ satisfies $|\Sigma|=\boA_1(M)$, then $\Sigma$ is of type $1$ or
$2$ and if $\Sigma\in\boU$ satisfies $2|\Sigma|=\boA_1(M)$, then $\Sigma$ is of type $3$.
\end{thm}

Actually in \cite{MaRo2}, the case (3) mentions the possibility for the orientable
$2$-sheeted cover to have index $0$ or $1$. In fact, the index $1$ case can be ruled out
thanks to the work of Ketover, Marques and Neves \cite{KeMaNe}.

If $\boS$ denotes the collection of all smooth embedded stable minimal surfaces, we define
$\boA_\boS(M)=\inf(\{|\Sigma|, \Sigma\in \boO\cap \boS\}\cup \{2|\Sigma|, \Sigma\in
\boU\cap \boS\})$. Actually we proved in \cite{MaRo2} that
$\boA_1(M)=\min(W_M,\boA_S(M))$. In order to simplify some notations, we will denote
$a(\Sigma)=|\Sigma|$ if $\Sigma\in \boO$ and $a(\Sigma)=2|\Sigma|$ is $\Sigma\in \boU$.

When $M$ is not compact, one can still define $\boO$ and $\boU$ for $M$ by considering
only compact embedded minimal surfaces in $M$. Of course these collections could be empty
but if its not $\boA_1(M)$ is well defined. If $M$ is a cusp manifold this can be done.


\subsection{The filler}\label{sec:filler}

We want to study $\boA_1(M)$ when $M$ is a cusp manifold. In order to do that the idea is
to change $M$ into a compact manifold $D(M)$ that contains all the compact minimal
surfaces of $M$. To do this the main tool are the fillers.

\begin{defn}\label{defn:filler}
Let $(T,d\sigma^2)$ be a flat torus and $L>10$ be a real number. A filler $F$ associated to
$T$ and $L$ is a solid torus endowed with a Riemannian metric $g$ with the following
properties.
\begin{itemize}
\item[$(i$)] Let $T_t$ be the set of points at distance $t$ from $\partial F$. For $t\in[0,L+1)$,
$T_t$ is a smooth flat torus and $T_{L+1}$ is a closed geodesic.
\item[$(ii)$] The diameter of $T_t$ is a decreasing function and the mean curvature vectors points
in the $\partial_t$ direction.
\item[$(iii)$] For $t\in [0,1]$, $T_t$ has the metric $e^{-2t}d\sigma^2$.
\item[$(iv)$] Any minimal surface $\Sigma$ that meets all the $T_t$ for $0\le t\le L-1$
has area at least $\kappa L$ where $\kappa$ is a constant depending on the systole of 
$(T,d\sigma^2)$.
\end{itemize}
\end{defn}

\begin{prop}\label{prop:filler}
Let $(T,d\sigma^2)$ be a flat torus and $L>10$. There exists a filler associated to $T$
and $L$. Moreover, let $\delta$ and $\s$ be the diameter and the systole of 
$(T,d\sigma^2)$ and $K\le \s/\delta$. Then there is $\delta_0>0$ that depends only on 
$K$ such that
\begin{itemize}
\item[$(v)$]if $\delta$ is less than
$\delta_0$, then any minimal surface $\Sigma$ with $\partial\Sigma\subset\partial F$ 
and index at most $1$ satisfies
either $\Sigma\cap T_1=\emptyset$ or $\Sigma\cap T_t\neq\emptyset$ for any $t\le L-1$.
\end{itemize}
\end{prop}

\begin{proof}
We construct $F$ as $T\times[L+1]$ with a Riemannian metric
which is singular on $T_{L+1}=T\times\{L+1\}$ in order for $T_{L+1}$ to be a geodesic. We
use the notation $F_t=T\times[t,L+1]$.

Let $f: [0,L+1]\to \R$ be a function satisfying the following properties
\begin{itemize}
\item $f(t)=t$ on $[0,1]$,
\item $f'>0$ on $[0,L+2/3)$ and $f'=0$ on $[L+2/3,L+1]$.
\item $f\le 3$.
\item $f'$ and $f''$ are bounded independently of $L$.
\end{itemize}
On $F\setminus F_L$, we define the metric $g=e^{-2f(t)}d\sigma^2+dt^2$. Since $f(t)=t$ on
$[0,1]$, $(iii)$ is satisfied.

In order to define the metric on $F_L$, we consider a well oriented orthonormal coordinate system on
$(T,d\sigma^2)$ such that $T$ is the quotient of $\R^2$ by the translations by $(\alpha,0)$ and
$(\beta,\ell)$.

Let $\eta:[0,1]\to [0,1]$ be a non increasing function such that $\eta$ is decreasing on
$[1/2,1]$, $\eta=1$ near $0$, and $\eta(x)=(1-x)\frac{2\pi}\alpha e^{f(L+1)}$ near $1$. On
$F_L$ we extend the definition of the metric $g$ by
$$
g=e^{-2f(t)}(\eta^2(t-L)dx_1^2+dx_2^2)+dt^2
$$
Since $\eta(1)=0$, the metric is singular at $t=L+1$. Let $D$ be the unit disk and consider
the solid torus $\T$ constructed as the quotient of $D\times[0,\ell]$ by the relation
$(p,0)\sim (R_{\beta}(p),\ell)$ where $R_{\beta}$ is the rotation of angle $\beta$. If
$(\rho,\theta)$ are the polar coordinates on $D$ and $h :\T\to F_L$ is the map
$(\rho,\theta, z)\mapsto (\frac\alpha{2\pi}\theta,v,L+1-\rho)$ the metric $h^*g$ is given by 
$$
h^*g=e^{-2f(L+1-\rho)}(\eta^2(1-\rho)\frac{\alpha^2}{4\pi^2}\theta^2+dz^2)+d\rho^2
$$
so, near $\rho=0$ (\textit{i.e.} $t=L+1$), it is equal to
$h^*g=\rho^2d\theta^2+e^{-2f(L+1)}dz^2+d\rho^2$ which is a smooth metric on $\T$ near the
core circle $\{\rho=0\}$. So $F$ is a smooth solid torus with a smooth metric and $(i)$ is satisfied.

Because of the monotonicity of $f$ and $\eta$, $(ii)$ is satisfied. Moreover the curvature
of $g$ is uniformly controlled on $F\setminus F_L$. If $\rho_0$ is the minimum of $1$ and
half the systole of $(T,d\sigma^2)$, then for any $p\in F_1\setminus F_{L-1}$ the geodesic
ball of center $\rho$ and radius $e^{-3}\rho_0$ is embedded in $F\setminus F_L$. 

Let $\Sigma$ be a minimal surface that meets all the $T_t$ for $t\in[0,L]$. Consider
$t_n=1+2e^{-3}\rho_0n$ and, for any $n\in\{0,\dots,n_0\}$ where $t_{n_0}\le L+2\le
t_{n_0+1}$, let $p_n$ be in $T_{t_n}$ with $p_n\in T_{t_n}\cap \Sigma$. Then by the
monotonicity formula, the area of $\Sigma$ in the ball of radius $e^{-3}\rho_0$ and center
$p_n$ is at least $ce^{-6}\rho_0^2$ for some universal constant $c$. Since these balls are
disjoint, the area of $\Sigma$ in $F\setminus F_L$ is at least
$$
(n_0+1)ce^{-6}\rho_0^2\ge \frac{L-3}2ce^{-3}\rho_0\ge \kappa L
$$
if $L\ge 10$ and where $\kappa$ only depends on the systole of $(T,d\sigma^2)$.

For item $(v)$, we notice that $(T,\frac1{\delta^2}d\Sigma^2)$ belongs to a 
compact subset of flat tori fixed by $K$. So Proposition~\ref{prop:max3} applies to 
$(F\setminus F_L,\delta^2e^{-2f(t)}\frac{d\sigma^2}{\delta^2}+dt^2)$ to prove that if 
$\Sigma$ has index at most $1$ and $\Sigma \subset T\times[0,,L-1]$ then 
$\Sigma\subset T\times [0,1]$.
\end{proof}


\subsection{The quantity $\boA_1(M)$ for cusp manifolds}\label{sec:boa1cusp}

In this section we recall the study of compact minimal surfaces inside orientable cusp manifolds we
made in \cite{CoHaMaRo,MaRo2}.

Let $M$ be a cusp manifold. First we prove that $M$ contains a compact embedded minimal
surface. Let $\eps$ be such that the $\eps$-thin part is only made of cusp ends. Since
$\partial M_{[\eps,\infty)}$ is smooth there is a homotopically closed family $\Lambda$ of
sweep-outs associated to a Morse function on $M_{[\eps,\infty)}$ (we recall that the tori components of $\partial M_{[\eps,\infty)}$ are leaves of the sweep-outs). If $\eps'<\eps$,
$M_{[\eps',\eps]}$ is foliated by tori that can be used to extend any continuous sweep-out
in $\Lambda$ into a sweep-out of $M_{[\eps',\infty)}$ that belongs to a homotopically 
closed
family $\Lambda'$. Since $W(M_{[\eps,\infty)},\partial M_{[\eps,\infty)},\Lambda)\ge
|\partial M_{[\eps,\infty)}|$ we obtain
$$
W(M_{[\eps,\infty)},\partial M_{[\eps,\infty)},\Lambda)\ge W(M_{[\eps',\infty)},\partial
M_{[\eps',\infty)},\Lambda')
$$
So there is $W_0>0$ such that $W_0\ge W(M_{[\eps,\infty)},\partial
M_{[\eps,\infty)},\Lambda)$ for any $\eps$. Besides a continuous sweep-out of
$M_{[\eps,\infty)}$ must sweep out also a fixed geodesic ball in $M$. So there is $w_0$
such that $W(M_{[\eps,\infty)},\partial M_{[\eps,\infty)},\Lambda)\ge w_0$ for any
$\eps$.

Thus we can choose $\eps$ small such that any flat tori $C$ in $\partial M_{[\eps,\infty)}$ has small diameter and
$w_0>|\partial M_{[\eps,\infty)}|$.
For each $C$, we consider a
filler $F^C$
associated to the flat torus $C$ and $L$ that will be chosen later. $\eps$ is chosen small
enough such that item $(v)$ of Proposition~\ref{prop:filler} is satisfied. Since there are a finite
number of $C$, item $(iv)$ of Definition~\ref{defn:filler} gives some constant $\kappa>0$
independent of $C$. Then $L$ is chosen such that $\kappa L\ge W_0+1$.

We can glue each
filler $F^C$ along $C$ to obtain a compact manifold without boundary denoted $D(M)$ with
some metric. The construction of $D(M)$ depends on two parameters $\eps$ and $L$, so
sometimes we will write $D_{\eps,L}(\barre M)$ (actually it also depends on the choice of
some coordinates on $F$). We
will use this construction in the following sections. $D(M)$ contains isometrically a $1$
tubular neighborhood of
$M_{[\eps,\infty)}$. Let $\{\Gamma_t\}_{t\in[0,1]}$ be a
continuous sweep-out of $M_{[\eps,\infty)}$ with $\LL(\Gamma_t)\le
W(M_{[\eps,\infty)},\partial M_{[\eps,\infty)},\Lambda)+1/2$. We can extend
$\{\Gamma\}_{t\in[0,1]}$ to a continuous sweep-out
$\{\widetilde\Gamma_t\}_{t\in[-L-1,1]}$ of $D(M)$ by considering
$\widetilde\Gamma_t=\cup_C \partial F_{-t}^C$ for $t\in[-L-1,0]$. Since, for $t<0$, $|\widetilde\Gamma_t|\le
|\partial M_{[\eps,\infty)}|$ we have $\LL(\widetilde\Gamma_t)\le \LL(\Gamma_t)$. By
Theorem~\ref{th:conttodis}, the width $W_{D(M)}$ is then less than $W_0+1/2$. Thus by
Theorem~\ref{th:leastarea}, there is a minimal surface $\Sigma$ in $D(M)$  
with index at 
most $1$ such that
$a(\Sigma)\le W_0+1/2$.

Because of our choice of $\eps$ and items $(iv)$ and $(v)$, if
$\Sigma$ enters in some $F_1^C$ then $a(\Sigma)\ge |\Sigma \cap F^C|\ge \kappa L\ge
W_0+1$; this is impossible. So $\Sigma$ stays outside of $F_1^C$
for any $C$ so $\Sigma$ is embedded in the part isometric to the $1$-tubular neighborhood
of $M_{[\eps,\infty)}\subset M$: $\Sigma$ is a compact minimal surface in $M$.

Now we know that $M$ contains compact embedded minimal surfaces and we can define the number
$\boA_1(M)$. In order to prove that $\boA_1(M)$ is realized as in
Theorem~\ref{th:leastarea}, we have the following argument. Let $S$ be a 
compact minimal surface in $M$. We construct $D(M)$ as
above with an extra hypothesis on $\eps$ which is $M_{[\eps,\infty)}$ contains $S$ and  
all compact
embedded minimal surfaces in $M$ with index at most $1$. The above 
construction gives $\boA_1(M)\le W_0+1/2$. 

Let $\Sigma$ is a minimal surface that realizes $\boA_1(D(M))$ (it has index at most $1$), 
we have $a(\Sigma)\le 
a(S)$. If $\Sigma$ enters in into $F^C_1$ for some $C$, we have
$$
a(\Sigma)\ge |\Sigma\cap F^C|\ge \kappa L\ge W_0+1\ge \boA_1(M)+1/2
$$
So $\Sigma$ does not enter into such a filler: $\Sigma$ belongs to $M$. Thus 
$\boA_1(D(M))=\boA_1(M)$ and 
$\boA_1(D(M))$ is realized by a minimal surface as in
Theorem~\ref{th:leastarea}.


The remainder of this paper is devoted to the study of the continuity of the $\boA_1$ functional
over the collection of orientable cusp manifolds. We are going to study the lower and the upper
semi-continuity of $\boA_1$.


\section{The upper semi-continuity study}\label{sec:upsemi}

In this section, we consider $(M_i)_i$ a sequence of cusp manifolds that converges to
$\barre M$ for the geometric convergence. The first step and the main step of the upper
semi-continuity study is to prove that the sequence $(\boA_1(M_i))_i$ is bounded. The
following proposition answers this question.

\begin{prop}\label{prop:upperbound}
Let $M_i\to \barre M$ be a converging sequence of cusp manifolds. Then for small $\eps$
and large $L$, $\limsup \boA_1(M_i)\le W_{D(\barre M)}$.
\end{prop}
	
\begin{proof}
The idea of
the proof consist in constructing a Riemannian manifold $(N_i,\tilde g_i)$ which is
$\kappa_i$-quasi isometric to $D(\barre M)$ with $\kappa_i\to 1$ and such that a large
part $N_i^1$ of $N_i$ is isometric to a large part $M_i$. Moreover $N_i$ is such that any
minimal surface with index at most $1$ that gets out of $N_i^1$ has area at 
least $W_{D(\barre M)}+1/2$. As a consequence,
a minimal surface $S_i$ in $N_i$ realizing $\boA_1(N_i)$ satisfies
$a(S_i)\le\kappa_i^2W_{D(\barre M)}<W_{D(\barre M)}+1/2$ for large $i$ and is contained in
$N_i^1$. Thus $\limsup \boA_1(M_i)\le \limsup a(S_i)\le \limsup \kappa_i^2W_{D(\barre
M)}=W_{D(\barre M)}$.

We choose $\eps$ small such that the $\eps$-thin part of $M$ is made only of cusp ends.
The convergence $M_i\to \barre M$ gives us  $\phi_i:\barre
M_{[\eps,\infty)}\to M_i$ as in Subsection~\ref{sec:conv}. From
Section~\ref{sec:boa1cusp}, we know that there is $W_0>0$ independent of $\eps$ and $L$
such that $W_{D_{\eps,L}(\barre M)}\le W_0+1/2$.

Let $C$ be one boundary component of $\barre M_{[\eps,\infty)}$ and $\barre A$ the part of
the $2$-tubular neighborhood of $C$ inside $\barre M_{[\eps,\infty)}$ (the rest of the
proof is written as
there is only one $C$ in $\partial \barre M_{[\eps,\infty)}$, actually we need to repeat the
argument for each $C$). $\barre A$ can be parametrized by $T\times[-2,0]$ with the
metric $\barre g=e^{-2x_3}(dx_1^2+dx_2^2)+dx_3^2$ where $(x_1,x_2)\in T$ are orthonormal
coordinates on $C$. 

By Subsection~\ref{sec:conv}, $\phi_i(\barre A)$ is a one sided neighborhood of $\phi_i(C)$ in
$\phi_i(\barre M)\subset M_i$. On $\phi_i(\barre A)$ we have the coordinates $T\times[-2,0]$ with the metric $g_i=a_{i,kl}dx_kdx_l$
 that $C^\infty$ converges to $\bar g$.

Let $N_i^1$ be $\phi_i(\barre M_{[\eps,\infty)}^1)$ with the metric $g_i$ where $\barre
M_{[\eps,\infty)}^1$ is the set of points in $M_{[\eps,\infty)}$ at distance at least $1$
from $\partial\barre M_{[\eps,\infty)}$. We notice that
$N_i^1$ contains the part of $\phi_i(\barre A)$ parametrized by $T\times[-2,-1]$. We are going to
modify the metric $g_i$ on $T\times[-1,0]$ in order to define a new metric $\tilde g_i$ on
$\phi_i(\barre M_{[\eps,\infty)})$ which will be the Riemannian manifold $N_i^2$.

Let $\chi:[-1,0]\to \R$ be $C^\infty$ such that $0\le\chi\le 1$, $\chi=1$ near $-1$ and
$\chi=0$ near $0$. We then define $\tilde g_i=\chi(x_3)g_i+(1-\chi(x_3))\bar g$. Since
$g_i$ and $\bar g$ are $C^\infty$ close. $\tilde g_i$ is also $C^\infty$ close to $\bar
g$. As explained above, $\tilde g_i$ turns $\phi_i(\barre M_{[\eps,\infty)})$ into a new
Riemannian manifold $(N_i^2,\tilde g_i)$. The map $\phi_i: \barre M_{[\eps,\infty)}\to
N_i^2$ is still well defined and since the metrics $\tilde g_i$ converge in the $C^\infty$
topology to $\bar g$.
$\phi_i$ is a $\kappa_i'$ quasi-isometry where $\kappa_i'\to 1$. Moreover $\phi_i$ is an
isometry close to $\partial \barre M_{[\eps,\infty)}$.

Let $L$ be large and consider a filler $F$ associated to $T$ and $L$. Since $N_i^2$ and
$\barre M_{[\eps,\infty)}$ are isometric close to their boundary we can glue to all of
them the filler $F$ to produce $(D_{\eps,L}(\barre M),\tilde g)$ and $(N_i,\tilde g_i)$
and extend the definition of
$\phi_i$ to a map $D(\barre M)\to N_i$ which is the identity on the filler. As a
consequence $\phi_i: D(\barre M)\to N_i$ is a $\kappa_i'$ quasi-isometry.

Let us estimate the area of a minimal surface $S\subset N_i$ with index at 
most $1$ that is not contained in
$N_i^1$. Thus $S$ must enter in some part of $N_i$ which is isometric to $T\times[-2,L]$
endowed with the
metric $\tilde g_i=\tilde a_{i,kl}dx_kdx_l$ which is $C^\infty$ close to $\tilde g=\bar g$ on
$T\times[-2,0]$ and is equal to $\tilde g=e^{-2f(x_3)}(dx_1^2+dx_2^2)+dx_3^2$ on
$T\times[0,L]$ ($f$ is introduced in Section~\ref{sec:filler}). Because of our choice of
$f$ function, the metrics $\tilde g_i$ and
$\tilde g$ satisfy the hypotheses H1, H2, H3 and H4 of Section~\ref{sec:maxprinciples} for a uniform
constant $A$. 

If $S$ does not meet all the tori $T_s$ for $s\in [-2,L]$ then, by
Proposition~\ref{prop:max3}, $S$ must stay outside of $T\times[-3/2,L+1]$, so $S\subset
N_i^1$. Since we assume $S_i\not\subset N_i^1$, it must meet all the tori $T_s$ for $s\in
[0,L]$. Then by Proposition~\ref{prop:filler}, $|S|\ge \kappa L$ for some $\kappa$
that only depends on the injectivity radius of $T_0$. Now, we choose $L$ large enough such
that $\kappa L> W_0+1$. We obtain $|S|\ge\kappa L>W_0+1 \ge W_{D(\barre M)}+1/2$. 
This
finishes the construction of $N_i$ and then $\limsup \boA_1(M_i)\le W_{D(\barre M)}$.
\end{proof}

We know that for $\eps$ small and $L$ large we have $\boA_1(\barre M)=\boA_1(D(\barre
M))=\min(\boA_\boS(\barre M),W_{D(\barre M)})$. So the above result gives us a first upper
semi-continuity property.

\begin{prop}\label{prop:upsemi}
Let $M_i\to \barre M$ be a converging sequence of cusp manifolds. If one of the
following hypotheses is satisfied then $\limsup \boA_1(M_i)\le \boA_1(\barre M)$.
\begin{itemize}
\item $\boA_1(\barre M)=W_{D(\barre M)}$
\item $\boA_1(\barre M)$ is realized by a stable non separating minimal surface $\Sigma$
\item $\boA_1(\barre M)$ is realized by a stable non degenerate minimal surface $\Sigma$
\end{itemize}
\end{prop}

\begin{proof}
The first case comes directly from the above proposition.

Let $\Sigma$ be a non separating minimal surface that realizes $\boA_1(\barre M)$. Let
$\eps$ be small such that $\Sigma$ is contained in the $\eps$-thick part of $\barre M$. 
Let $\phi_i : \barre M_{[\eps,\infty)}\to M_i$ be the $\kappa_i$ quasi-isometry
associated to the convergence $M_i\to \barre M$. Then $\phi_i(\Sigma)$ is a surface in
$M_i$ with $ a(\phi_i(\Sigma))\le \kappa_i^2a(\Sigma)$. Because of the topology of the
$\eps$-thin part (cusps or solid tori), $\phi_i(\Sigma)$ is non separating in $M_i$.
So taking a small $\eps_i$ and a large $L_i$ we can see $\phi_i(\Sigma)$ as a non separating
surface in $D_{\eps_i,L_i}(M_i)$. So minimizing the area in the non vanishing homology class of
$\Sigma$ there is a minimal surface $S_i$ in $D(M_i)$ with $a(S_i)\le
a(\phi_i(\Sigma))\le \kappa_i^2a(\Sigma)=\kappa_i^2\boA_1(\barre M)$. Thus
$$
\boA_1(M_i)= \boA_1(D(M_i))\le a(S_i)\le \kappa_i^2\boA_1(\barre M)
$$
and this gives the result.

Concerning the last case, as above, let $\eps$ be small such that $\Sigma$ is contained 
in the $\eps$-thick part of $\barre M$ and $\phi_i:\barre
M_{[\eps,\infty)}\to M_i$. Let $h_i=\phi_i^* g_i$. Since $M_i\to \barre M$, the metrics
$h_i$ converge in the $C^\infty$ topology to $\barre g$. Since $\Sigma$ is a non
degenerate surface, for large $i$, $\Sigma$ can be deformed to a minimal surface $S_i$ in
$(\barre M_{[\eps,\infty)},h_i)$. So $\phi_i(S_i)$ is a minimal surface in $M_i$ and 
$$
\limsup \boA_1(M_i)\le \limsup a(S_i)=a(\Sigma)=\boA_1(\barre M)
$$ 
\end{proof}

\begin{remarq}
We notice that the hypothesis $\boA_1(\barre M)=W_{D(\barre M)}$ is satisfied if
$\boA_1(\barre M)$ is realized by an index $1$ minimal surface.

The second case is realized if $\boA_1(\barre M)$ is realized by a non orientable minimal
surface.
\end{remarq}


\section{The lower semi-continuity study}\label{sec:lowsemi}

In this section we are going to prove that the $\boA_1$ functional is lower
semi-continuous. 





\subsection{An exclusion property}

Let $S$ be a two-sided embedded surface. Let $\nu$ be a choice of a unit normal vectorfield
along $S$ and $f:S\to \R$ be a smooth function. Then we can define
$$
\exp_{S,f}: S\to M; p\mapsto \exp_p(f(p)\nu(p))
$$
If $f$ is sufficiently small, $\exp_{S,f}(S)$ is an embedded surface which inherits from
$S$ a natural unit normal vector still denoted by $\nu$. The lemma below is inspired
by Lemma~16 in \cite{Son}

\begin{lem}\label{lem:increasing}
Let $S$ be a two-sided embedded surface and $U$ be a subset of $S$ such
that the mean curvature of $S$ vanishes on $U$. If $S\setminus U$ has non empty
interior, there is a positive function $f$ and $\tau>0$ such that $\exp_{S,sf}(S)$
has positive mean curvature on $\exp_{S,sf}(U)$ with respect to the naturally induced unit
normal vector for $0<s<\tau$.
\end{lem}

\begin{proof}
Let $S$ and $U$ be as in the statement. Let $q$ be a function on $S$ such that
$q=Ric(\nu,\nu)+\|A\|^2$ on $U$ and $\boL=-\Delta-q$ has negative first eigenvalue on $S$.
It is enough to assume $q$ is large enough somewhere in $\Sigma\setminus U$ to ensure that
the first eigenvalue $\lambda_1$ is negative. Let $f>0$ be the first eigenfunction of $\boL$. Consider
$S_t=\exp_{S,tf}(S)$ and $H_t(p)$ be the mean curvature of $S_t$ at
$\exp_{S,tf}(p)$. It is known that $2\partial_t{H_t}_{|t=0}=\Delta
f+(Ric(\nu,\nu)+\|A\|^2)f=-\lambda_1 f+(Ric(\nu,\nu)+\|A\|^2-q)f>0$ on $U$. Thus, there is
$\tau>0$ such that $H_t(p)>0$ for any $t\in(0,\tau]$ and $p\in U$.
\end{proof}

Using the above lemma we can prove the following result.

\begin{prop}\label{prop:outside}
Let $A_0$, $\delta_0$ and $\s_0\le 1$ be positive. Then there is $\ell_0$ and $\barre R$ such the
following is true. Let $\ell\le \ell_0$ and $M$ be a cusp manifold such that
$\boA_1(M)\le A_0$ and $M$ contains a tubular end $N_{R_\ell}$ of a geodesic loop
of length $\ell$ and such that $S_{R_\ell}$ has diameter less than $\delta_0$ and systole
larger than $\s_0$. Let $\Sigma$ be an embedded minimal surface that realizes $\boA_1(M)$
then $\Sigma\cap N_{R_\ell-\barre R}$ is empty.
\end{prop}

\begin{proof}
We first assume that $\Sigma$ is stable. This implies that there is $k_0$ such that
$\Sigma$ has curvature bounded by $k_0$. So by
Corollary~\ref{cor:estimarea1}, there is $\ell_0$ and $\barre R$ such that, if $\ell\le
\ell_0$, either $\Sigma\cap N_{R_\ell-\barre R}$ is empty or $\Sigma\cap N_{R_\ell-\barre R}$
has area at least $2\pi(\cosh(R_\ell-\barre R)-1)$. Actually, if $\ell_0$ is chosen such
that
$2\pi(\cosh(R_{\ell_0}-\barre R)-1)\ge A_0$, the second case can not occur and $\Sigma\cap
N_{R_\ell-\barre R}$ is empty.

So we can assume that $\Sigma$ is separating and has index $1$. By
Proposition~\ref{prop:estimarea}, there is $\barre R$, $\ell_0$ and $\kappa$ such that
$\Sigma\cap N_{R_\ell-\barre R}=\emptyset$ or $|\Sigma\cap N_R|\ge \kappa \s_0
e^{R-R_\ell}$ for any $R\in[3,R_\ell-\barre R]$. Moreover $\ell_0$ and $\barre R$ can be
chosen such that the preceding paragraph is still true.

Let us assume that $\Sigma\cap N_{R_\ell-\barre R}$ is not empty. Let us notice that $|S_R|=\pi
\ell\sinh(2R)\le \sqrt 3\frac{\sinh(2R)}{\sinh (2R_\ell)}\le \sqrt3 e^{2(R-R_\ell)}$. So choosing
$R_*\ge \barre R$ such that $e^{-R_*}\le \frac{\kappa\s_0}{4\sqrt 3}$ we obtain
$$
|\Sigma\cap N_R|\ge 4|S_R|.
$$
for any $3\le R\le  R_\ell-R_*$.

Let $\eps$ be small such that the $\eps$-thin part of $M$ contains only cusp ends. For $L$
large we consider $D_{\eps,L}(M)=D(M)$ such that $\boA_1(M)=\boA_1(D(M))$. So $\Sigma$ is
a separating index $1$ minimal surface in $D(M)$ that realizes $\boA_1(D(M))$. The idea is
now to construct a discrete sweep-out $S$ of $D(M)$ such that $\LL(S)<|\Sigma|$ which
contradicts $\Sigma$ realizes $\boA_1(M)$. We notice that $N_{R_\ell}$ is still
isometrically included in $D(M)$.

$\Sigma$ separates $M$ into two connected components $\Ome_1$ and $\Ome_2$. Let $R\in
[R_\ell-R_*-1,R_\ell-R_*]$ such that $S_R$ is transverse to $\Sigma$. We define
$\Gamma=\Sigma\cap S_R$. The subset $\Ome_i\cap N_R$ has mean convex boundary made of
pieces of $S_R$ and $\Sigma$. We can
find a least area minimal surface $\Sigma_i\subset \Ome_i\cap N_R$ with
$\partial\Sigma_i=\Gamma$ and homologous to $S_R\cap\Ome_i$. Since $S_R\cap\Ome_i$ is a
surface with boundary $\Gamma$, $|\Sigma_i|\le|S_R\cap\Ome_i|\le |S_R|$. Besides
$\Sigma_i\cup(S_R\cap\Ome_i)$ bounds a subset $D_i$ of $N_R\cap \Ome_i$. By boundary
regularity of solutions to the Plateau problem~\cite{HaSi}, $\Sigma_i$ is a smooth surface up to its
boundary $\Gamma$ and, by the maximum principle, $\Sigma_i$ is transverse to $\Sigma$
along $\Gamma$. We notice that since $\Sigma_i$ is smooth up to its boundary we can
slightly extend $\Sigma_i$ across $\Gamma$ to $\Sigma_i'$. $\Sigma_i'$ is not assumed to
be minimal outside of $\Sigma_i$ and $\partial\Sigma_i'$ is assumed to be outside
$\Ome_i$.

Let us fix $i\in\{1,2\}$ and consider $\nu$ the unit normal along $\Sigma$ pointing into
$\Ome_i$. Since $\Sigma$ has index $1$, there is $\tau>0$ and $f_i>0$ on $\Sigma$ such
that $\exp_{\Sigma,sf_i}(\Sigma)$ is an embedded surface with positive mean curvature for
any $s\in(0,\tau]$. Moreover, we assume
$f_i>1$. If $\nu_i$ denote the unit normal along $\Sigma_i'$ pointing into $D_i$ along $\Sigma_i$, by Lemma~\ref{lem:increasing}, there is $g_i>0$ on $\Sigma_i'$ such
that $\exp_{\Sigma_i',sg_i}(\Sigma_i')$ is an embedded surface with positive mean curvature on
$\exp_{\Sigma_i',sg_i}(\Sigma_i)$ for any $s\in(0,\tau]$. Moreover, we assume
$g_i< 1$ and $\exp_{\Sigma_i',sg_i}(\partial\Sigma_i')\not\subset \Ome_i$ for any
$s\in[0,\tau]$.

Let us define $U_{i,0}=D_i\cup (\Ome_i\setminus N_R)$. For $s\le \tau$ we define
\begin{gather*}
V_{i,s}= \Big(\bigcup_{0\le s'\le s} \exp_{\Sigma,s'f_i}(\Sigma))\cup \bigcup_{0\le s'\le s}
\exp_{\Sigma_i',s'g_i}(\Sigma_i')\Big) \cap U_{i,0},\\
U_{i,s}=U_{i,0}\setminus V_{i,s}
\end{gather*}

We postpone the precise description of $\partial U_{i,s}$ to the end of the proof.
Actually we are going to prove that there is a smaller $\tau$ such that, for $s\in[0,\tau]$, $\partial U_{i,s}$
is
$\exp_{\Sigma,sf_i}(A_{i,s})\cup \exp_{\Sigma_i',sg_i}(B_{i,s})$ where $A_{i,s}$ is a smooth
subdomain in $\Sigma$ and $B_{i,s}$ is a smooth subdomain of $\Sigma_i'$. Moreover both
components of $\partial U_{i,s}$ are transverse. We also have $B_{i,s}\subset \Sigma_i$
and $A_{i,s}\subset
\Sigma\setminus N_{R-1}$. $s\mapsto \partial U_{i,s}$ is then a continuous map with values
in $\boZ_2(M)$. Moreover $\MM(\partial U_{i,s})\le \MM(\partial U_{i,0})$. We then have
$$
\MM(\partial U_{i,s})\le \MM(\partial U_{i,0})=|\Sigma_i|+|\Sigma\setminus N_R|\le
|S_R|+|\Sigma|-|\Sigma\cap N_R|\le |\Sigma|-3|S_R|
$$
Besides we notice that, since $A_{i,s}\subset \Sigma\setminus N_{R-1}$ and $B_{i,s}\subset
\Sigma_i'$, $\partial U_{i,s}$ is piecewise smooth mean convex in the sense of Definition~10 in~\cite{Son}.

Using the work of Song in \cite{Son}, we can adapt the work of the authors in \cite{MaRo2}
to the case $\partial U_{i,\tau}$ is not smooth and prove the following statement. 
\begin{claim}
There is a continuous sweep-out $\{\partial U_{i,s}\}_{s\in[\tau ,1]}$ of $U_{i,\tau}$ such that 
$$
\LL(\{\partial U_{i,s}\}_{s\in[\tau ,1]})\le |\Sigma|-2|S_R|.
$$
\end{claim}

\begin{proof}
Because of the Appendix in~\cite{Son}, we know that there exists a homotopically closed
family $\Lambda$ of sweepouts in $U_{i,\tau}$. Let us assume that $W(U_{i,\tau},\partial
U_{i,\tau},\Lambda)\ge |\Sigma|-5/2|S_R|>|\partial U_{i,\tau}|$.

Thus by Theorem~12 in~\cite{Son}, there is a closed minimal surface $S$ in $U_{i,\tau}$.
As in the proof of Lemma~20 in~\cite{Son}, $N=U_{i,\tau}\setminus S$ is then a mean convex
subset such that $\partial U_{i,\tau}$ has non vanishing homology class in $N$. Thus we
can minimize the area in this homology class in order to get $S'$ a stable minimal surface
such that $|S'|\le |\partial U_{i,\tau}|$. But this implies $\boA_\boS(D(M))\le |\partial
U_{i,\tau}|\le |\Sigma|-3|T_R|<\boA_1(D(M))\le \boA_\boS(D(M))$ which is contradictory. So
$W(U_{i,\tau},\partial U_{i,\tau},\Lambda)\le |\Sigma|-5/2|S_R|$ and there is $\{\partial
U_{i,s}\}_{s\in[\tau ,1]}$ with 
$$\LL(\{\partial U_{i,s}\}_{s\in[\tau ,1]})\le |\Sigma|-2|S_R|.$$
\end{proof}


Using these two sweep-outs, we can define a family $G_s$ (see Figure~\ref{fig:fig3}) of open
subsets of $M$ by
$$
G_s=\begin{cases}
U_{1,1-s}&\text{if }0\le s\le 1\\
U_{1,0}\cup N_{s-1}&\text{if } 1\le s\le R+1\\
(\Ome_1\cup N_R)\setminus (\barre{D_2\setminus N_{2R+1-s}})&\text{if }R+1\le s \le 2R+1\\
M\setminus\barre{U_{2,s-2R-1}}&\text{if }2R+1\le s\le 2R+2
\end{cases}
$$

\begin{figure}[h]
\begin{center}
\resizebox{0.9\linewidth}{!}{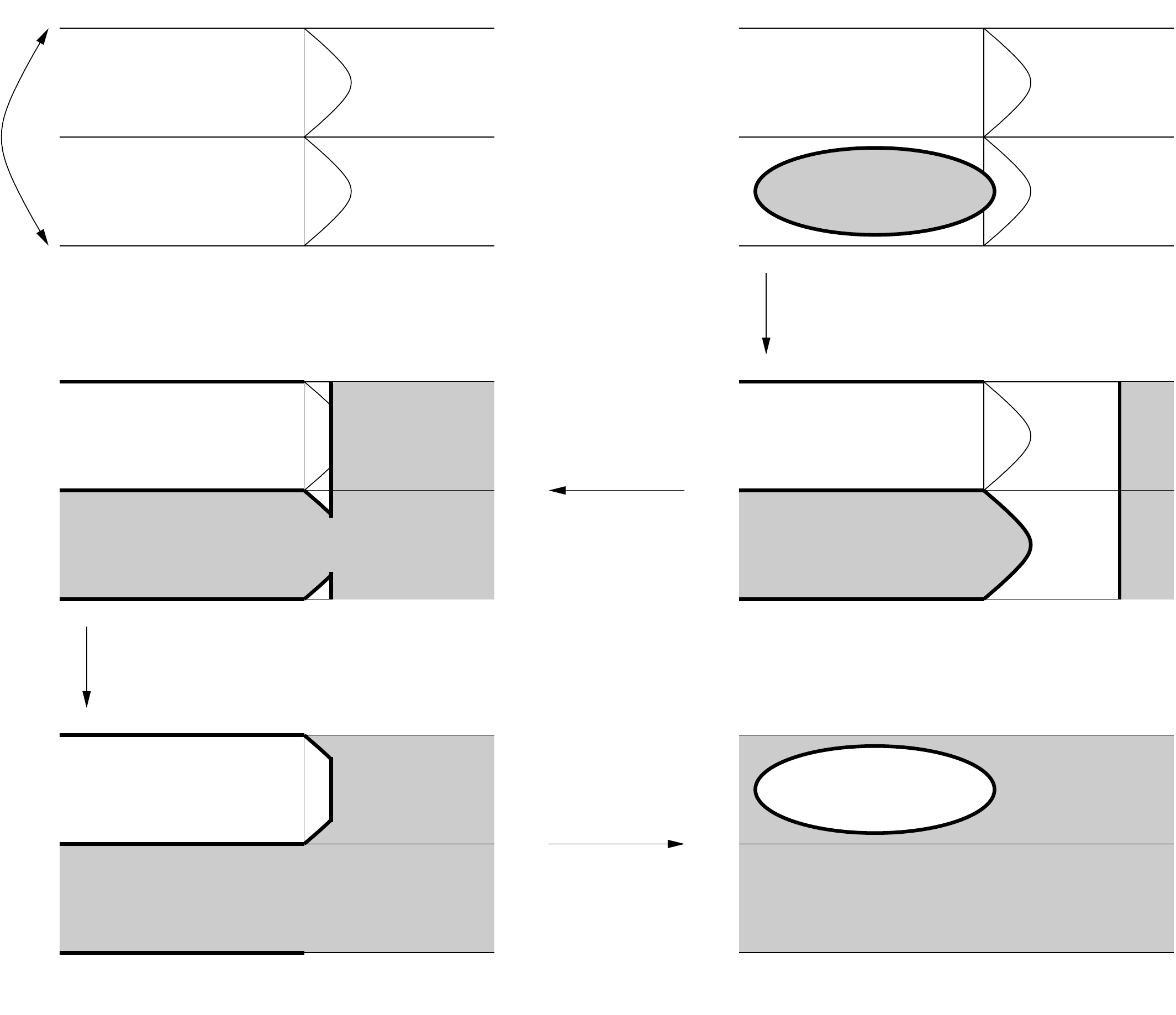}
\caption{A schematic view of $G_s$ (the bottom and top line are identified)\label{fig:fig3}}
\end{center}
\end{figure}

The family $G_s$ satisfies $\boH^3(G_s\triangle G_t)\to 0$ as $s\to t$. Moreover, $\partial
G_s$ is rectifiable so $\Phi: s\mapsto \partial G_s$ is a continuous path in $\boZ_2(M)$
for the flat topology. Moreover $G_0=\emptyset$ and $G_{2R+2}=M$.
Let us now study $\MM(\Phi(s))$. For $s\in[0,1]$ we have $\MM(\Phi(s))=\MM(\partial
U_{1,1-s})\le |\Sigma|-2|S_R|$. For $s\in[1,R+1]$, $\partial G_s$ is contained in
$\partial U_{1,0}\cup S_{s-1}$ so $\MM(\Phi(s))\le \MM(\partial U_{1,0})+|S_{s-1}|\le
|\Sigma|-2|S_R|$. If $s\in[R+1,2R+1]$, $\partial U_s$ is contained in $\partial
U_{2,0}\cup S_{2R+1-s}$ so $\MM(\Phi(s))\le \MM(\partial U_{2,0})+|S_{2R+1-s}|\le
|\Sigma|-2|S_R|$. Finally for $s\in[2R+1,2R+2]$, $\MM(\Phi(s))=\MM(\partial
U_{2,s-2R-1})\le |\Sigma|-2|S_R|$.

After a reparametrization, we have then constructed a continuous map $\Phi:[0,1]\to
\boZ_2((D(M)),\boF)$ satisfying all the hypotheses of Theorem~\ref{th:conttodis} with
$\sup\MM(\Phi)\le |\Sigma|-2|T_R|<|\Sigma|$. So by Theorem~\ref{th:conttodis}, we have
$\boA_1(M)=\boA_1(D(M))\le W_{D(M)}\le\sup\MM(\Phi)\le
|\Sigma|-2|T_R|<|\Sigma|=\boA_1(M)$. This gives a contradiction with $\Sigma\cap
N_{R_L-\barre R}\neq \emptyset$ and finishes the proof.

Let us come back to the study of $\partial U_{i,s}$ for small $s$ and check the properties
we announced. Clearly this boundary
is contained in $\Sigma_{i,s}'=\exp_{\Sigma_i',sg_i}(\Sigma_i')$ and
$\Sigma_s=\exp_{\Sigma,sf_i}(\Sigma)$. We need to understand the intersection of these two
surfaces when $s$ is small. 

We define $F_i:(p,s)\in\Sigma\times(-\tau,\tau)\mapsto \exp_{\Sigma,sf_i}(p)\in M$ and
$G_i:(p,s)\in\Sigma_i'\times(-\tau,\tau)\mapsto \exp_{\Sigma_i',sg_i}(p)\in M$ for small $\tau$.
The map $F_i$ defines a smooth coordinate system in a neighborhood $N$ of $\Sigma$. Let us write
$F_i^{-1}=(P,T): N\to \Sigma\times  (-\tau,\tau)$. Let us remark that at a point $p\in
\Sigma$, ${DF_i^{-1}}_{|p}:X\in T_pM\mapsto (\pi_p(X),(X,\nu)/f_i(p))\in T_p\Sigma\times
\R$ where $\pi_p$ is the normal projection $T_pM\to T_p\Sigma$.

Let $V$ be a neighborhood of $\Gamma$ inside $\Sigma_i'$ contained in $N$. There is $\tau'$ such that
$G_i(V\times(-\tau',\tau'))\subset N$. Let $\eta_i$ be the conormal to
$\Gamma$ in $\Sigma_i'$ pointing to $\Sigma_i$. So neighboring points to $\Gamma$ in
$\Sigma_i'$ can be parametrized by
$(p,t)\in\Gamma\times(-\eps,\eps)\mapsto\exp_p^i(t\eta_i(p))$ where $\exp^i$ is the
exponential map in $\Sigma_i'$. Thus such a point
has image by $G_i(\cdot,s)$ in the intersection
$F_i(\Sigma,s)\cap G_i(\Sigma_i',s)$ ($s$ small) if
$$
L_i(p,s,t):=T(G_i(\exp_p^i(t\eta_i(p)),s))-s=0
$$
Solving $t$ as a function of $(p,s)\in \Gamma\times \R$ can be done near
$\Gamma\times\{0\}$ using the implicit function theorem since $L_i(p,0,0)=0$. Indeed we have
$\partial_t L_i(p,0,0)=(\nu(p),\eta_i(p))/f_i(p)>0$ since both $\nu$ and $\eta_i$ point to
$\Ome_i$. So $t_i(p,s)$ can be defined near $\Gamma\times\{0\}$. At $(p,0)$ we also have
$$
0=\partial_s(L(p,s,t_i(p,s))=\frac{(\nu,\eta_i)}{f_i}\partial_st_i+\frac{g_i(\nu,\nu_i)}{f_i}-1
$$
Thus $\partial_st_i=\frac{f_i}{(\nu,\eta_i)}(1-\frac{g_i(\nu,\nu_i)}{f_i})>0 $
since $g_i/f_i<1$. The curve $\gamma_{i,s}(p)=\exp_p^i(t_i(p,s)\eta(p))$ is sent by
$G_i(\cdot,s)$ on the intersection $F_i(\Sigma,s)\cap G_i(\Sigma_i',s)$:
$\beta_i(\cdot,s)=G_i(\gamma_{i,s}(\cdot),s)$ is a parametrization of the intersection.
Moreover $\gamma_{i,s}$
bounds a subdomain $B_{i,s}$ in $\Sigma_i'$ whose image by $G_i(\cdot,s)$ is the piece of
$\partial U_{i,s}$ contained in $G_i(\Sigma_i',s)$.
Since $\partial_st_i>0$, we have $B_{i,s}\subset \Sigma_i$. At $(p,0)$, we also have
$$
\partial_s\beta_i=g_i\nu_i+\frac{f_i}{(\nu,\eta)}(1-\frac{g_i(\nu,\nu_i)}{f_i})\eta_i
$$

The curve $\gamma_s(\cdot)=P(\beta_i(\cdot,s))$ on $\Sigma$ is such that
$F_i(\gamma_s(\cdot),s)$ is also a parametrization of the intersection. $\gamma_s$ bounds a
subdomain $A_{i,s}$ in $\Sigma$ such that $\partial U_{i,s}$ is
$F_i(A_{i,s},s)\cup G_i(B_{i,s},s)$. We notice that, at $s=0$, 
$$
\partial_s\gamma_s(p)=\pi_p(\partial_s\beta_i)=\pi_p(g_i\nu_i+\frac{f_i}{(\nu,\eta)}
(1-\frac{g_i(\nu,\nu_i)}{f_i})\eta_i)
$$
We notice that since $B_{i,s}\subset \Sigma_i$, the mean curvature of
$G_i(B_{i,s},s)$ is positive ($s>0$). The same is true for the mean curvature of
$F_i(A_{i,s},s)$. Moreover both surfaces intersect at an angle less than
$\pi$. Finally using the first variation formula and $\Sigma$ and $\Sigma_i$
are minimal, we have at $s=0$
\begin{align*}
\partial_s(|F_i(A_{i,s},s)|+|G_i(B_{i,s},s)|&=
-\int_{\Gamma}(\partial_s\gamma_s(p),\eta)+(\partial_s\gamma_{i,s}(p),\eta_i)\\
&=-\int_\Gamma(g_i\nu_i+\frac{f_i}{(\nu,\eta_i)}
(1-\frac{g_i(\nu,\nu_i)}{f_i})\eta_i,\eta)\\
&\qquad\qquad+\frac{f_i}{(\nu,\eta_i)}(1-\frac{g_i(\nu,\nu_i)}{f_i})\\
&=-\int_\Gamma (g_i\nu_i+\frac{f_i}{(\nu,\eta_i)}
(1-\frac{g_i(\nu,\nu_i)}{f_i})\eta_i,\eta+\eta_i)\\
&=-\int_\Gamma (\partial_s\beta_i,\eta+\eta_i)
\end{align*}
where $\eta$ is the unit conormal to $\Gamma$ in $\Sigma$ that points outside $N(R)$. We
notice that for $s>0$, $\beta_s$ is inside $U_{i,0}$ so $\partial_s\beta_s$ points to
$U_{i,0}$ and is orthogonal to the tangent space to $\Gamma$. $\eta+\eta_i$ is a vector
that bisects the wedge corresponding to $U_{i,0}$ and contained in the orthogonal to the
tangent space to $\Gamma$. So $(\partial_s\beta_i,\eta+\eta_i)>0$ along $\Gamma$ and
$\partial_s(|F_i(A_{i,s},s)|+|G_i(B_{i,s},s)|<0$. This implies that $|\partial
U_{i,s}|<|\partial U_{i,0}|$ for $s>0$ small. So all the stated properties are satisfied. 
\end{proof}

\subsection{The lower semi-continuity} 
We have the following result.

\begin{prop}\label{prop:losemi}
Let $M_i\to \barre M$ be a converging sequence of cusp manifolds. Then $\liminf
\boA_1(M_i)\ge \boA_1(\barre M)$.
\end{prop}

\begin{proof}
Let us consider a minimal surface $\Sigma_i$ in $M_i$ such that $a(\Sigma_i)=\boA(M_i)$.
By Proposition~\ref{prop:upperbound}, we know that there is $A_0$ such that
$|\Sigma_i|<A_0$. Moreover, by Corollary~\ref{cor:diamsys}, there is $\ell_0$, $\delta_0$ and $\s_0$ such that if
$M_i$ contains a geodesic loop of length $\ell\le\ell_0$ then its tubular neighborhood
$N_{R_\ell}$ satisfies $S_{R_\ell}=\partial N_{R_\ell}$ has diameter less than $\delta_0$ and systole larger
than $\s_0$. Thus there is $\barre R$ such that $\Sigma_i\cap N_{R_\ell-\barre R}$ is
empty by Proposition~\ref{prop:outside}.

This implies that there is $\eps>0$ such that $\Sigma_i\subset {M_i}_{[\eps,\infty)}$ for
any $i$. Since $M_i\to \barre M$ there is $\phi_i:\barre M_{[\eps/2,\infty)}\to M_i$ which
is a $\kappa_i$ quasi-isometry where $\kappa_i\to 1$. Moreover we have
${M_i}_{[\eps,\infty)}\subset\phi_i(\barre M_{[\eps/2,\infty)})$. Let $\tilde g_i=\phi_i^*
g_i$ and $\widetilde \Sigma_i=\phi_i^{-1}(\Sigma_i)\subset \barre M_{[\eps/2,\infty)}$
which is a minimal surface for the metric $\tilde g_i$. Since $M_i\to \barre M$ we have
$\tilde g_i\to \bar g$ in the $C^\infty$ topology. Moreover $a_{\tilde g_i}(\widetilde
\Sigma_i)\le A_0$ and the index of $\Sigma_i$ is $0$ or $1$.

Thus we can apply the compactness result of Sharp (Theorem~A.6 in~\cite{Shar}). It implies
that there is a closed connected embedded minimal surface $\barre\Sigma$ in $(\barre
M,\bar g)$ such that $(\widetilde \Sigma_i)$ converges in the varifold sense to $\barre
\Sigma$ with some multiplicity. Moreover, the convergence is smooth outside a finite
number of points. If $\barre \Sigma$ is orientable, then 
$$
\boA_1(\barre M)\le a_{\bar g}(\barre \Sigma)=|\barre \Sigma|_{\bar g}\le\lim
|\widetilde\Sigma_i|_{\tilde g_i}\le\liminf \boA_1(M_i)
$$

If $\barre \Sigma$ is non-orientable, then either $\widetilde\Sigma_i$ is non orientable
or $\widetilde\Sigma_i$ is orientable and the convergence must be with multiplicity at
least $2$. In both cases, we have
$$
\boA_1(\barre M)\le a_{\bar g}(\barre \Sigma)=2|\barre \Sigma|_{\bar g}\le\lim
a_{\tilde g_i}(\widetilde\Sigma_i)=\liminf \boA_1(M_i)
$$
So the proposition is proved.
\end{proof}

 
\appendix


\section{}\label{sec:appen}


\subsection{A uniform graph lemma}\label{sec:appendix}

Let us consider $\R^3$ endowed with the metric $\bar g=h^2(x_3)(dx_1^2+dx_2^2)+dx_3^2$.
For $k_1,k_2,k_3,k_4\in \{1,2,3\}$ and $p\le 4$, we recall that
$$
n_p(k_1,\dots,k_p)=\#\{i\in\{1,\dots,p\}|k_i\in\{1,2\}\}.
$$ 
We consider a second metric $g=a_{kl}(x_1,x_2,x_3)dx_kdx_l$. We assume that there is
some $A$ such that the following hypotheses occurs
\begin{itemize}
\item[H1] $\frac1{A^2}\bar g\le g\le A^2\bar g$
\item[H2] $\frac{|h'|}h\le A$ and $\frac{|h''|}h\le A$.
\item[H3] $|a_{kl}|\le A h^{n_2(k,l)}(x_3)$, $|\partial_ia_{kl}|\le A h^{n_3(k,l,i)}(x_3)$ and
$|\partial_i\partial_j a_{kl}|\le A h^{n_4(k,l,i,j)}(x_3)$.
\end{itemize}
We notice that the metric $\bar g$ satisfies also the hypotheses of the last item.

\begin{lem}\label{lem:estimgraph}
Let $\bar g$ and $g$ as above and consider $\eps_0$, $k_0$, then there is $C>0$ such that
the following is true.
Let $\Sigma$ be a surface in $\R^2\times [a,b]$ endowed with the metric $g$ which is
tangent to $\R^2\times\{\bar t\}$ at $\bar p=(0,0,\bar t)$ such that $d_\Sigma(\bar p,
\partial\Sigma)\ge \eps_0$ and $|A_\Sigma|\le k_0$.

Then there is a function $u$ defined on the disk $\{(x_1,x_2)\in \R^2|x_1^2+x_2^2\le
2C^2/h^2(\bar t)\}$ such that $(x_1,x_2)\mapsto(x_1,x_2,\bar t+u(x_1,x_2))$ is a
parametrization of a neighborhood of $\bar p$ in $\Sigma$. Moreover $u$ satisfies
$$
|u|\le A\eps_0,\quad \|\nabla u\|\le h(\bar t)\quad\textrm{and}\quad\|\Hess u\|\le \frac1C h^2(\bar t)
$$
\end{lem}

\begin{proof}
First we replace $\Sigma$ by the geodesic disk of center $\bar p$ and $\eps_0$. Since
$a_{33}\ge \frac1{A^2}$, the distance between $\{x_3=\bar t\}$ and $\{x_3=\bar t\pm t\}$ is at
least $t/A$. So $\Sigma$ is
contained in $\R^2\times[\bar t-A\eps_0,\bar t+A\eps_0]$. Let us also remark that since
$\frac{|h'|}h\le A$ we have $e^{-A|x_3-\bar t|}h(\bar t)\le h(x_3)\le e^{A|x_3-\bar
t|}h(\bar t)$. Then $e^{-A^2\eps_0} h(\bar t)\le h(x_3)\le e^{A^2\eps_0}h(\bar t)$ on
$[\bar t-A\eps_0,\bar
t+A\eps_0]$.

Let us consider $\Psi : \R^2\times [\bar t-A\eps_0,\bar t+A\eps_0]\to \R^2\times[\bar t-A\eps_0,\bar
t+A\eps_0]: (y_1,y_2,y_3)\mapsto (\frac1{h(\bar t)}y_1,\frac1{h(\bar t)}y_2,y_3)$. Then
the metric $g^*=\Psi^*g$ can be written $b_{kl}(y_1,y_2,y_3)dy_kdy_l$ where $b_{kl}=h(\bar
t)^{-n_2(k,l)}a_{kl}\circ \Psi$. Thus 
$$
|b_{kl}|= h(\bar
t)^{-n_2(k,l)}|a_{kl}|\circ \Psi \le A\frac{h^{n_2(k,l)}(y_3)}{h^{n_2(k,l)}(\bar t)}\le
Ae^{n_2(k,l)A^2\eps_0}\le Ae^{2A^2\eps_0}
$$
So there is a constant $B$ such that $|b_{kl}|\le B$. A similar computation proves that
$|\nabla b_{kl}|\le B$ and $|\Hess b_{kl}|\le B$. Using Hypothesis H1, we also have
$\frac1{A^2}\Psi^*\bar g\le g^*\le A^2\Psi^*\bar g$ where
$\Psi^*\bar g=\frac{h^2(y_3)}{h^2(\bar t)}(dy_1^2+dy_2^2)+dy_3^2$. This implies that $\det
g^*$ is far from $0$ and $\infty$. So the coefficients $b^{kl}$ of the inverse of $g^*$ satisfy
$|b^{kl}|\le B$ and, for any $k\in\{1,2,3\}$, $\frac1B\le b_{kk}\le B$ and $\frac1B\le b^{kk}\le B$.

Let us define $\Sigma^*=\Psi^{-1}(\Sigma)$, $\Sigma^*\subset (\R^3,g^*)$ is a geodesic
disk of radius $\eps_0$ and curvature bounded by $k_0$. Let us consider
$g_e=dy_1^2+dy_2^2+dy_3^2$ the Euclidean metric. Because of the the control we have
on $g^*$, there is $\eps_1$ that depends only on $\eps_0$, $A$ and $B$ and $k_1$ that
depends only on $k_0$, $A$ and $B$ such that $(\Sigma^*, g_e)$ has curvature bounded by
$k_1$ and $d_{\Sigma^*,g_e}(\bar p, \partial\Sigma^*)\ge \eps_1$ (the proof of this
result can be found in the Appendix of \cite{RoSoTo} more precisely see the proof of
Propositions~4.1 and 4.3).

So we have a surface in the Euclidean space $\R^3$ with curvature bounded by $k_1$,
$d_{\Sigma^*,g_e}(\bar p, \partial\Sigma^*)\ge \eps_1$ and that is tangent to
$\R^2\times\{\bar t\}$ at $\bar p$. Then a classical uniform graph lemma (see
Proposition~2.3 in \cite{RoSoTo}) implies that there is $C$ that depends only on
$k_1$ and $\eps_1$ such the following is true. There is a function $u$ defined on the Euclidean disk of
radius $\sqrt2C$ centered at the origin such that $(y_1,y_2)\mapsto (y_1,y_2,\bar
t+u(y_1,y_2))$ is a parametrization of a neighborhood of $\bar p$ in $\Sigma^*$. Moreover
$$
|u|\le 2 C,\quad|\nabla u|\le 1\quad\text{and}\quad \|\Hess u\|\le \frac1C
$$

In order to come back to the original coordinate system we define the function
$v(x_1,x_2)=u(h(\bar t)x_1,h(\bar t)x_2)$ which is defined on $\{(x_1,x_2)\in
\R^2|x_1^2+x_2^2\le 2\delta^2/h^2(\bar t)\}$ and satisfies
$$
|v|\le 2C,\quad|\nabla v|\le h(\bar t)\quad\text{and}\quad \|\Hess v\|\le
\frac1C h^2(\bar t)
$$
We notice that, since $\Sigma\subset \R^2\times [\bar t-A\eps_0,\bar t+A\eps_0]$, we have
$|v|\le A\eps_0$.
\end{proof}

\subsection{The minimal surface equation}

Several times we consider graphs that are minimal surfaces; let us write the equation
solved by these graphs.

On $\R^3$ we consider the metric $g=a_1^2(x_3)dx_1^2+a_2^2(x_3)dx_2^2+dx_3^2$ which is a
model for the metric in cusp or tubular ends. Let $u$ be a function in a domain of $\R^2$
and consider the graph parametrized by $X(x_1,x_2)=(x_1,x_2,u(x_1,x_2))$. The induced
metric is 
$$
(a_1^2(u)+u_{x_1}^2)dx_1^2+2u_{x_1}u_{x_2}dx_1dx_2+(a_2^2(u)+u_{x_2}^2)dx_2^2
$$
So the area element is $Wdx_1dx_2=(a_1^2(u)a_2^2(u)+a_2^2(u)u_{x_1}^2+a_1^2(u)u_{x_2}^2)dx_1dx_2$.

So if $v$ is an other function with zero boundary values and $A(t)$ is the area of the
graph of $u+tv$, the derivative of $A$ at $t=0$ is 
\begin{align*}
A'(0)&=\int\frac1W\Big(a_1(u)a_1'(u)a_2^2(u)v+a_1^2(u)a_2(u)a_2'(u)v+a_2(u)a_2'(u)u_{x_1}^2v\\
&\qquad\qquad+a_2^2(u)u_{x_1}v_{x_1}+a_1(u)a_1'(u)u_{x_2}^2v+a_1^2(u)u_{x_2}v_{x_2}\Big)\\
&=\int\frac{a_1(u)a_2(u)}W\big(a_1'(u)a_2(u)+a_1(u)a_2'(u)+\frac{a_2'(u)}{a_1(u)}u_{x_1}^2
+\frac{a_1'(u)}{a_2(u)}u_{x_2}^2\big)v\\
&\qquad\qquad-v\Div\frac{(a_2^2(u)u_{x_1},a_1^2(u) u_{x_2})}W
\end{align*}

Thus the graph is minimal if $u$ satisfies
\begin{multline}\label{eq:mse}
0=\Div\frac{(a_2^2(u)u_{x_1},a_1^2(u)
u_{x_2})}W\\-\frac{a_1(u)a_2(u)}W\big(a_1'(u)a_2(u)+a_1(u)a_2'(u)+\frac{a_2'(u)}{a_1(u)}u_{x_1}^2
+\frac{a_1'(u)}{a_2(u)}u_{x_2}^2\big)
\end{multline}

\bibliographystyle{plain}
\bibliography{../reference.bib}

\end{document}